\documentclass[12pt]{article}

\usepackage{amssymb,amsmath,amsthm}
\usepackage{times}

\newcommand{\dist}{\mathop{\mathrm{dist}}\nolimits}

\newcommand{\osc}{\mathop{\mathrm{osc}}}

\theoremstyle{plain}
\newtheorem{theorem}{Theorem}[section]
\newtheorem{lemma}[theorem]{Lemma}
\newtheorem{corollary}[theorem]{Corollary}
\newtheorem{lemmaprime}{Lemma}
\theoremstyle{definition}
\newtheorem{remark}[theorem]{Remark}
\newtheorem{definition}[theorem]{Definition}
\newtheorem{example}[theorem]{Example}

\numberwithin{equation}{section}
\title{%
 A gradient estimate for solutions
 to parabolic equations
 with discontinuous coefficients%
}
\author{%
 Jishan Fan\thanks{%
  Department of Applied Mathematics, 
  Nanjing Forestry University, 
  Nanjing 210037, P.\ R.\ China. 
 }, 
 Kyoungsun Kim\thanks{%
  Department of Mathematics, Ewha Womans University, 
  Seoul 120-750, Korea.
 }, 
 Sei Nagayasu\thanks{%
  Department of Mathematical Science, 
  Graduate School of Material Science, 
  University of Hyogo, 
  2167 Shosha, Himeji, Hyogo 671-2280, 
  Japan. 
  e-mail: sei[at]sci.u-hyogo.ac.jp 
 } \\ and Gen Nakamura\thanks{
  Department of Mathematics, Hokkaido University,
  North 10, West 8, Kita-Ku, Sapporo, Hokkaido 060-0810, 
  Japan. 
  e-mail: gnaka[at]math.sci.hokudai.ac.jp
 }%
}
\date{\empty}
\begin{document}
\maketitle
\begin{abstract}
Li-Vogelius and Li-Nirenberg gave
a gradient estimate for solutions of strongly elliptic equations and systems of divergence forms
with piecewise smooth coefficients, respectively. The discontinuities of the coefficients are assumed to be given by
manifolds of codimension 1, which we called them {\it manifolds of discontinuities}. Their gradient estimate is independent
of the distances between manifolds of discontinuities.
In this paper, we gave
a parabolic version of their results.
That is, we gave a gradient estimate
for parabolic equations of divergence forms
with piecewise smooth coefficients. The coefficients are assumed to be independent of time and their discontinuities are likewise the
previous elliptic equations.
As an application of this estimate,
we also gave a pointwise gradient estimate
for the fundamental solution
of a parabolic operator
with piecewise smooth coefficients. The both gradient estimates are independent
of the distances between manifolds of discontinuities.
\end{abstract}
\section{Introduction.}\label{section:introduction}

For strongly elliptic, second order scalar equations with real coefficients, it is well-known that their solutions have
the H{\"o}lder continuity even in the case that the coefficients are only bounded measurable functions.
However, the solutions do not have the Lipschitz continuity in general.
For example,
Piccinini-Spagnolo~\cite[p.\ 396, Example 1]{PiccininiSpagnolo}
and
Meyers~\cite[p.\ 204]{Meyers}
gave the following example:
\begin{example}{\rm (\cite{Meyers}, \cite{PiccininiSpagnolo})}
Let $B_1:=\{x \in \mathbb{R}^{n} : \lvert x \rvert < 1\}$ 
and each $a_{i j} \in L^\infty (B_1)$ be defined as
\[
a_{1 1}
= \frac{M x_{1}^{2} + x_{2}^{2}}{\lvert x \rvert^{2}} , \quad
a_{2 2} = \frac{x_{1}^{2} + M x_{2}^{2}}{\lvert x \rvert^{2}} , \quad
a_{1 2} = a_{2 1} =
\frac{(M-1) x_{1} x_{2}}{\lvert x \rvert^{2}}
\]
with a constant $M>1$. 
Then, if we define $u$ as
\begin{equation}\label{eq:example of u}
u (x) =
\lvert x \rvert^{1 / \sqrt{M}}
\frac{x_{1}}{\lvert x \rvert},
\end{equation}
it is easy to see that the H{\"o}lder exponent of $u$ is 
at least less than 
or equal to $1 / \sqrt{M}$ 
(indeed, for $\overline{x} = ( x_{1} , 0 )$ we have 
\begin{math}
\lvert u( \overline{x} ) - u(0) \rvert 
= \lvert \overline{x} \rvert^{1 / \sqrt{M}}
\end{math}. Hence we have 
\[
\frac{\lvert u( \overline{x} ) - u(0) \rvert}{
 \lvert \overline{x} \rvert^{( 1 / \sqrt{M} ) + \varepsilon}
} = \lvert \overline{x} \rvert^{- \varepsilon}
\rightarrow + \infty \mbox{ as } \overline{x} \rightarrow 0
\]
for any $\varepsilon > 0$.)
and $u$ satisfies the strongly elliptic scalar equation with real coefficients
\begin{equation}\label{eq:elliptic equation}
\sum_{i,j=1}^{2}
\frac{\partial}{\partial x_{i}} \left(
 a_{i j} \frac{\partial u}{\partial x_{j}}
\right) = 0.
\end{equation}

The same thing can be said also to the parabolic equation
\begin{equation}\label{eq:parabolic equation}
\frac{\partial u}{\partial t}
- \sum_{i,j=1}^{2}
\frac{\partial}{\partial x_{i}} \left(
 a_{i j} \frac{\partial u}{\partial x_{j}}
\right) = 0,
\end{equation}
because $u$ given by \eqref{eq:example of u} satisfies this equation.
\end{example}
This example shows that we cannot expect gradient estimates of solutions to equations \eqref{eq:elliptic equation} and \eqref{eq:parabolic equation}
in the case $a_{ij}\in L^\infty(B_1)$, but we may have the estimates in the case of piecewise $C^\mu$ (see \eqref{eq:piecewiseCmu} below) coefficients.

The fact that the gradient estimate of solutions is independent of the distances between manifolds of discontinuities was first observed by
Babu{\v{s}}ka-Andersson-Smith-Levin~\cite{Babuska}
numerically
for certain homogeneous isotropic linear systems of elasticity,
that is $\lvert \nabla u \rvert$
is bounded independently of the distances between manifolds of discontinuities. They considered that this numerical property of solutions is mathematically true. This is
the so-called Babu{\v{s}}ka's conjecture. Recently, \cite{LiVogelius} and \cite{LiNirenberg}
gave mathematical proofs for this conjecture. In elasticity, a small static deformation of an elastic medium with inclusions can be described by an elliptic system of divergence form with piecewise smooth coefficients. The discontinuities of coefficients form the boundaries of inclusions.
Similar physical interpretation is also possible for heat conductors.
Our main theorem \ref{theorem:main} given below ensures that
this property also holds for parabolic equations of the form \eqref{eq:parabolic equation}. The details of result given in \cite{LiVogelius} and \cite{LiNirenberg} for scalar equations will be given below as Theorem \ref{theorem:LN}.

In order to state our main theorem, we begin with introducing several notations which will be used throughout this paper.
Let $D \subset \mathbb{R}^{n}$ be a bounded domain
with a $C^{1, \alpha}$ boundary for some $0 < \alpha < 1$, which means
that the domain $D$ contains $L$ disjoint subdomains
$D_{1} , \ldots , D_{L}$
with $C^{1, \alpha}$ boundaries, i.e.
$D = ( \bigcup_{m=1}^{L} \overline{D_{m}} ) \setminus \partial D$,
and we also assume that
$\overline{D_{m}} \subset D$ for $1 \leq m \leq L-1$. Physically, $D$ is a material and $D_m\,(1\le m\le L-1)$ are
considered as inclusions in $D$.
We define the $C^{1, \alpha}$ norm 
(resp.\ $C^{1, \alpha}$ seminorm)
of $C^{1, \alpha}$ domain $D_{m}$
in the same way as in~\cite{LiNirenberg},
that is, as the largest positive number $a$
such that in the $a$-neighborhood of every point of
$\partial D_{m}$, identified as $0$
after a possible translation and rotation
of the coordinates so that $x_{n} = 0$ is the tangent to
$\partial D_{m}$ at $0$,
$\partial D_{m}$ is given by the graph of a $C^{1, \alpha}$ function $\psi_{m}$,
defined in $\lvert x^{\prime} \rvert < 2 a$
($x^{\prime} = ( x_{1} , \ldots , x_{n-1} )$),
the $2 a$-neighborhood of $0$ in the tangent plane, and it satisfies the estimate
\begin{math}
\lVert \psi_{m} \rVert_{
 C^{1, \alpha} ( \lvert x^{\prime} \rvert < 2 a )
} \leq 1 / a
\end{math}
(resp.\ 
\begin{math}
[ \psi_{m} ]_{
 C^{1, \alpha} ( \lvert x^{\prime} \rvert < 2 a )
} \leq 1 / a
\end{math}), 
where 
\begin{align*}
[ \psi ]_{
 C^{1, \alpha} ( \lvert x^{\prime} \rvert < 2 a )
} 
& := \sup_{
 \lvert x^{\prime} \rvert , \lvert \xi^{\prime} \rvert < 2 a
} \frac{\lvert
 \nabla^{\prime} \psi ( x^{\prime} )
 - \nabla^{\prime} \psi ( \xi^{\prime} ) 
\rvert}{\lvert x^{\prime} - \xi^{\prime} \rvert^{\alpha}} , 
\displaybreak[1] \\
\lVert \psi \rVert_{
 C^{1, \alpha} ( \lvert x^{\prime} \rvert < 2 a )
} 
& := \lVert \psi \rVert_{C^{1} ( \lvert x^{\prime} \rvert < 2 a )}
+ [ \psi ]_{
 C^{1, \alpha} ( \lvert x^{\prime} \rvert < 2 a )
} . 
\end{align*}

Further, let $( a_{i j} )$ be a symmetric, positive definite matrix-valued function
defined on $D$ satisfying
\begin{equation}\label{eq:aijxiixij}
\lambda \lvert \xi \rvert^{2}
\leq \sum_{i,j=1}^{n} a_{i j} (x) \xi_{i} \xi_{j}
\leq \Lambda \lvert \xi \rvert^{2} .
\end{equation}
Here each $a_{i j}$ is piecewise $C^{\mu}$ in $D$,
$0 < \mu < 1$, that is
\begin{equation}\label{eq:piecewiseCmu}
a_{i j} (x) = a_{i j}^{(m)} (x) \mbox{ for } x \in D_{m} , \
1 \leq m \leq L
\end{equation}
with $a_{i j}^{(m)} \in C^{\mu} ( \overline{D_{m}} )$.

As we have already mentioned above, we will discuss in this paper a gradient estimate
for solutions to parabolic equations
with piecewise smooth coefficients.
Our result is a parabolic version
for the results of Li-Vogelius~\cite{LiVogelius} and the scalar equations version of Li-Nirenberg~\cite{LiNirenberg}.
They showed that solutions $u\in H^1(D)$
to the elliptic equation
\begin{equation}\label{eq:elliptic}
\sum_{i,j=1}^{n} \frac{\partial}{\partial x_{i}}
\left( a_{i j} \frac{\partial u}{\partial x_{j}} \right)
= h + \sum_{i=1}^{n} \frac{\partial g_{i}}{\partial x_{i}},
\end{equation}
where $h\in L^\infty(D)$ and each $g_i$ is defined in $D$ such that $g_i|_{D_m}\,\,(1\le m\le L)$ have continuous extensions $\in C^\mu(\overline{D_m}),\,0<\mu<1$ up to $\partial D_m$ have global $W^{1, \infty}$
and piecewise $C^{1, \alpha^{\prime}}$ estimates (see \eqref{eq:ellipticest} below).
These estimates are independent of the distances
between inclusions
when a material has inclusions.

\medskip
We first give the result of Li-Nirenberg~\cite{LiNirenberg} for scalar equations.

\begin{theorem}[{\cite[Theorem~1.1]{LiNirenberg}}]\label{theorem:LN}
For any $\varepsilon > 0$,
there exists a constant $C_{\sharp} > 0$ such that
for any $\alpha^{\prime}$ satisfying
\[
0 < \alpha^{\prime}
< \min \left\{ \mu , \frac{\alpha}{2 ( \alpha + 1 )} \right\} ,
\]
we have
\begin{equation}\label{eq:ellipticest}
\sum_{m=1}^{L} \lVert u \rVert_{
 C^{1, \alpha^{\prime}} ( \overline{D_{m}} \cap D_{\varepsilon} )
} \leq C_{\sharp} \left(
 \lVert u \rVert_{L^{2} (D)}
 + \lVert h \rVert_{L^{\infty} (D)}
 + \sum_{m=1}^{L} \sum_{i=1}^{n} \lVert g_{i} \rVert_{
  C^{\alpha^{\prime}} ( \overline{D_{m}} )
 }
\right) ,
\end{equation}
where we denote
\[
D_{\varepsilon} := \{
 x \in D : \dist ( x, \partial D ) > \varepsilon
\}
\]
and a positive constant $C_{\sharp}$ depends only on
\begin{math}
n, L, \mu, \alpha , \varepsilon , \lambda , \Lambda ,
\lVert a_{i j} \rVert_{C^{\alpha^{\prime}} ( \overline{D_{m}})}
\end{math}
and the $C^{1, \alpha^{\prime}}$ norms of $D_{m}$.
\end{theorem}

\begin{remark}
The constant $C_{\sharp} > 0$
is independent of the distances between inclusions $D_{m}$.
Therefore, the estimate (\ref{eq:ellipticest}) holds
even in the case that some of inclusions touch another inclusions
as in Figure~\ref{figure:1}. 
\end{remark}

\begin{figure}
\begin{center}
\begin{picture}(240,120)
\qbezier(48,15.6024)(96,3.6)(176,3.6)
\qbezier(176,3.6)(240,3.6)(240,45.5976)
\qbezier(240,45.5976)(240,75.6)(192,93.6)
\qbezier(192,93.6)(144,111.6)(80,111.6)
\qbezier(80,111.6)(0,111.6)(0,69.6024)
\qbezier(0,69.6024)(0,27.5976)(48,15.6024)
\qbezier(34.6,47.8012)(44.2,41.8)(60.2,41.8)
\qbezier(60.2,41.8)(73,41.8)(73,62.7988)
\qbezier(73,62.7988)(73,77.8)(63.4,86.8)
\qbezier(63.4,86.8)(53.8,95.8)(41,95.8)
\qbezier(41,95.8)(25,95.8)(25,74.8012)
\qbezier(25,74.8012)(25,53.7988)(34.6,47.8012)
\put(49,64){\makebox(0,0){$D_{1}$}}
\qbezier(102.72,38.24096)(95.04,33.44)(82.24,33.44)
\qbezier(82.24,33.44)(72,33.44)(72,50.23904)
\qbezier(72,50.23904)(72,62.24)(79.68,69.44)
\qbezier(79.68,69.44)(87.36,76.64)(97.6,76.64)
\qbezier(97.6,76.64)(110.4,76.64)(110.4,59.84096)
\qbezier(110.4,59.84096)(110.4,43.03904)(102.72,38.24096)
\put(92,56){\makebox(0,0){$D_{2}$}}
\qbezier(158.72,73.75904)(151.04,78.56)(138.24,78.56)
\qbezier(138.24,78.56)(128,78.56)(128,61.76096)
\qbezier(128,61.76096)(128,49.76)(135.68,42.56)
\qbezier(135.68,42.56)(143.36,35.36)(153.6,35.36)
\qbezier(153.6,35.36)(166.4,35.36)(166.4,52.15904)
\qbezier(166.4,52.15904)(166.4,68.96096)(158.72,73.75904)
\put(148,44){\makebox(0,0){$D_{3}$}}
\qbezier(154.997793341,67.4460357714)(153.9798989874,71.8587214973)(149.4544155878,76.3842048969)
\qbezier(149.4544155878,76.3842048969)(145.8340288681,80.0045916166)(139.8946713174,74.0652340658)
\qbezier(139.8946713174,74.0652340658)(135.651691219,69.8222539674)(135.8213968465,64.5613795154)
\qbezier(135.8213968465,64.5613795154)(135.9911024739,59.3005050634)(139.6114891936,55.6801183437)
\qbezier(139.6114891936,55.6801183437)(144.1369725932,51.1546349441)(150.0763301439,57.0939924948)
\qbezier(150.0763301439,57.0939924948)(156.0167059284,63.0343682793)(154.997793341,67.4460357714)
\put(144,64){\makebox(0,0){$D_{4}$}}
\qbezier(206.976,74.303616)(200.832,76.224)(190.592,76.224)
\qbezier(190.592,76.224)(182.4,76.224)(182.4,69.504384)
\qbezier(182.4,69.504384)(182.4,64.704)(188.544,61.824)
\qbezier(188.544,61.824)(194.688,58.944)(202.88,58.944)
\qbezier(202.88,58.944)(213.12,58.944)(213.12,65.663616)
\qbezier(213.12,65.663616)(213.12,72.384384)(206.976,74.303616)
\put(200,66.4){\makebox(0,0){$D_{5}$}}
\qbezier(206.976,41.1643392)(200.832,43.4688)(190.592,43.4688)
\qbezier(190.592,43.4688)(182.4,43.4688)(182.4,35.4052608)
\qbezier(182.4,35.4052608)(182.4,29.6448)(188.544,26.1888)
\qbezier(188.544,26.1888)(194.688,22.7328)(202.88,22.7328)
\qbezier(202.88,22.7328)(213.12,22.7328)(213.12,30.7963392)
\qbezier(213.12,30.7963392)(213.12,38.8612608)(206.976,41.1643392)
\put(200,31.68){\makebox(0,0){$D_{6}$}}
\put(148,20){\makebox(0,0){$D_{7}$}}
\end{picture}
\end{center}
\caption{The case that an inclusion touches another inclusion.
($L=7$)\label{figure:1}}
\end{figure}
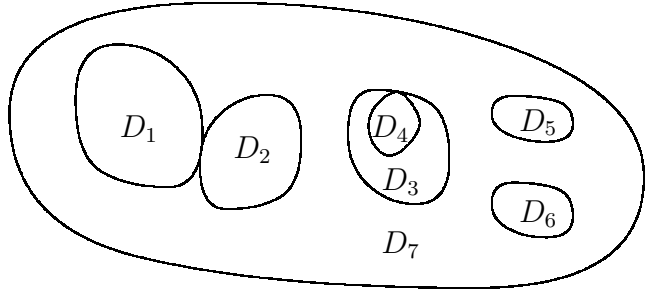%

Now, we consider the parabolic equation
\begin{equation}\label{eq:parabolic}
\frac{\partial u}{\partial t}
- \sum_{i,j=1}^{n} \frac{\partial}{\partial x_{i}}
\left( a_{i j} \frac{\partial u}{\partial x_{j}} \right)
= f - \sum_{i=1}^{n} \frac{\partial f_{i}}{\partial x_{i}}
\mbox{ in } Q,
\end{equation}
where
\begin{align*}
& f\in L^\infty(Q),\,\frac{\partial f}{\partial t}\in L^\kappa(Q), 
\displaybreak[1] \\
& f_{i}\in L^p(Q),\,\frac{\partial f_i}{\partial t}\in L^p(Q) \mbox{ and }
f_{i} = f_{i}^{(m)} \mbox{ on } D_{m} \times (0,T],
\end{align*}
with $p>n+2$, $\kappa = p(n+2)/(n+2+p)$,
$Q := D \times (0,T]$
and
$f_{i}^{(m)} \in L^{\infty} ( 0, T; C^{\mu} ( \overline{D_{m}} ) ) $.

Now we define a weak solution to the equation (\ref{eq:parabolic}). 
\begin{definition}
We call 
\begin{math}
u \in V_{2}^{1,0} (Q) 
:= L^{2} ( 0,T; H^{1} (D) ) 
\cap C ( [0,T]; L^{2} (D) )
\end{math}
a weak solution to the equation (\ref{eq:parabolic}) 
when 
\begin{align}
& \int_{D} u( x, t^{\prime} ) \, \varphi ( x, t^{\prime} ) \, d x 
- \int_{0}^{t^{\prime}} \int_{D}
 u( x, t ) \, \frac{\partial \varphi}{\partial t} ( x, t ) \, 
d x \, d t \notag \\
& \mbox{}
+ \int_{0}^{t^{\prime}} \int_{D}
 \sum_{i,j=1}^{n} a_{i j} (x) \, 
 \frac{\partial u}{\partial x_{j}} (x,t) \, 
 \frac{\partial \varphi}{\partial x_{i}} (x,t) \, 
d x \, d t \notag \\
& = \int_{0}^{t^{\prime}} \int_{D}
 f(x,t) \, \varphi (x,t) \, 
d x \, d t 
+ \int_{0}^{t^{\prime}} \int_{D}
 \sum_{i=1}^{n} f_{i} (x,t) \, 
 \frac{\partial \varphi}{\partial x_{i}} (x,t) \, 
d x \, d t \label{eq:weaksolution}
\end{align}
for any 
\begin{math}
\varphi \in L^{2} ( 0, T; \mathring{H}^{1} (D) )
\cap H^{1} ( 0, T; L^{2} (D) )
\end{math}
with $\varphi ( \cdot , 0 ) = 0$
and $0 < t^{\prime} \leq T$. 
\end{definition}

Our main result is as follows.

\begin{theorem}[Main theorem]\label{theorem:main}
Any weak solutions $u \in V_{2}^{1,0} (Q)$ 
to {\rm (\ref{eq:parabolic})}
have the following up to the inclusion boundary regularity estimate:
For any $\varepsilon > 0$,
there exists a constant $C_{\sharp}^{\prime} > 0$
such that for any $\alpha^{\prime}$ satisfying
\begin{equation}\label{eq:alphaprime}
0 < \alpha^{\prime}
< \min \left\{ \mu , \frac{\alpha}{2 ( \alpha + 1 )} \right\} ,
\end{equation}
we have
\[
\sum_{m=1}^{L}
\sup_{\varepsilon^{2} < t \leq T}
\lVert u ( \cdot , t ) \rVert_{
 C^{1, \alpha^{\prime}} ( \overline{D_{m}} \cap D_{\varepsilon} )
} \leq C_{\sharp}^{\prime}
\left(
 \lVert u \rVert_{L^{2} (Q)} + F_{\ast} + F_{\ast \ast}
\right) ,
\]
where
\begin{align*}
F_{\ast}
& := \lVert f \rVert_{L^{\kappa} (Q)}
+ \lVert f \rVert_{
 L^{\max \{ 2, \kappa \}} (Q)
} + \lVert f \rVert_{L^{\infty} (Q)}
+ \left\lVert
 \frac{\partial f}{\partial t}
\right\rVert_{L^{\kappa} (Q)} , \displaybreak[1] \\
F_{\ast \ast}
& := \sum_{i=1}^{n} \Biggl(
 \lVert f_{i} \rVert_{L^{p} (Q)}
 + \left\lVert
  \frac{\partial f_{i}}{\partial t}
 \right\rVert_{L^{2} (Q)} 
 + \left\lVert
  \frac{\partial f_{i}}{\partial t}
 \right\rVert_{L^{p} (Q)} \\
 & \hspace*{20ex} \mbox{}
 + \sum_{m=1}^{L} \sup_{0 < t \leq T}
 \lVert f_{i} ( \cdot , t ) \rVert_{
  C^{\alpha^{\prime}} ( \overline{D_{m}} )
 }
\Biggr)
\end{align*}
and
$C_{\sharp}^{\prime}$ depends only on
\begin{math}
n, L, \mu, \alpha , \varepsilon , \lambda , \Lambda , p,
\lVert a_{i j} \rVert_{C^{\alpha^{\prime}} ( \overline{D_{m}})}
\end{math}
and the $C^{1, \alpha^{\prime}}$ norms of $D_{m}$.
\end{theorem}
\begin{remark}\label{remark:main}
(i)
Again,
the constant $C_{\sharp}^{\prime} > 0$
is independent of the distances between inclusions $D_{m}$.
Then Theorem~\ref{theorem:main} holds
even in the case that an inclusion touches another inclusion
as Figure~\ref{figure:1}.

(ii)
It is easy to obtain
\begin{align*}
F_{\ast}
& \leq C^{\ast}
\left(
 \lVert f \rVert_{L^{\infty} (Q)}
 + \left\lVert
  \frac{\partial f}{\partial t}
 \right\rVert_{L^{\kappa} (Q)}
\right) , \displaybreak[1] \\
F_{\ast \ast}
& \leq C^{\ast} \sum_{i=1}^{n} \left(
 \sum_{m=1}^{L} \sup_{0 < t \leq T}
 \lVert f_{i} ( \cdot , t ) \rVert_{
  C^{\alpha^{\prime}} ( \overline{D_{m}} )
 }
 + \left\lVert
  \frac{\partial f_{i}}{\partial t}
 \right\rVert_{L^{p} (Q)}
\right) .
\end{align*}
However, a constant $C^{\ast} > 0$
depends on $T$ and $D$, unfortunately.
\end{remark}

For heat conductive materials with inclusions, \eqref{eq:parabolic} describes the temperature distribution in the materials. When these inclusions are unknown and need to be identified, thermography is one of non-destructive testing which identifies these inclusions. The measurement for the thermography could be temperature distribution at the boundary generated by injecting heat flux at the boundary.
The mathematical analysis for this thermography has not yet been developed so far. However, if we have enough measurements, the so called dynamical probe method (\cite{IKN}) can give a mathematically rigorous way to identify these inclusions. In the proof of justifying this method, the gradient estimate of the fundamental solution of parabolic equation with non-smooth coefficient is one of the essential ingredients.

The dynamical probe method has been developed only for the case that the inclusions do not touch another inclusions. So, it is natural to consider the case when some of them touch. For the first task to handle this case, we need to have the gradient estimate of the fundamental solution. Our main result has given the answer to this. Similar situation can be considered for stationary thermography and
non-destructive testing using acoustic waves. For example, \cite{NUW} and \cite{Yoshida}
effectively used a result of Li-Vogelius~\cite{LiVogelius} to give a procedure of reconstructing inclusions
by enclosure method (see \cite{Ikehata}, for example). What is interested about their arguments is that, by adding further arguments, we can
even reconstruct the inclusions in the case that they can touch another inclusions (\cite{Nag-Nak}). Therefore, we believe that
our gradient estimates will be useful for inverse problems identifying unknown inclusions.

The rest of this paper is organized as follows.
In Section~\ref{section:proof},
we prove our main theorem, i.e.\ Theorem~\ref{theorem:main}
by applying Lemma~\ref{lemma:keyestimate}.
We prove Lemma~\ref{lemma:keyestimate}
in Section~\ref{section:someestimates}.
In Section~\ref{section:gradestoffundsol},
we consider a pointwise gradient estimate
for the fundamental solution
of parabolic operators
with piecewise smooth coefficients
by applying Theorem~\ref{theorem:main}.
\section{Proof of main result.}\label{section:proof}
In this section, we prove our main theorem.
We first state some estimates in Lemma~\ref{lemma:keyestimate}
which we need to prove our main theorem.
We prove Lemma~\ref{lemma:keyestimate}
in Section~\ref{section:someestimates}.
\begin{lemma}\label{lemma:keyestimate}
Let $( a_{i j} )$ be a matrix-valued function defined on $D$.
Assume that $( a_{i j} )$ is symmetric, positive definite,
and satisfies the condition {\rm (\ref{eq:aijxiixij})}.
Let $Q$ as before and
\begin{math}
\widehat{Q}_{\varepsilon}
:= D_{\varepsilon} \times ( \varepsilon^{2} , T ]
\end{math}.
Then for $p>n+2$,
a weak solution $u \in V_{2}^{1,0} (Q)$ 
to {\rm (\ref{eq:parabolic})} satisfies
the following estimates:
\begin{align}
\sup_{\varepsilon^{2} < t \leq T}
\lVert u ( \cdot , t ) \rVert_{L^{2} ( D_{\varepsilon} )}
& \leq C \left( \lVert u \rVert_{L^{2} (Q)} + F_{0} \right) ,
\label{eq:LinftyL2} \displaybreak[1] \\
\lVert u \rVert_{L^{\infty} ( \widehat{Q}_{\varepsilon} )}
& \leq C \left( \lVert u \rVert_{L^{2} (Q)} + F_{0} \right) ,
\label{eq:Linfty} \displaybreak[1] \\
\left\lVert \frac{\partial u}{\partial t} \right\rVert_{
 L^{2} ( \widehat{Q}_{\varepsilon})
}
& \leq C \left( \lVert u \rVert_{L^{2} (Q)} + F_{1} \right) ,
\label{eq:dudtL2uL2}
\end{align}
where we set
\begin{align}
F_{0} & :=
\lVert f \rVert_{L^{\frac{p(n+2)}{n+2+p}} (Q)}
+ \sum_{i=1}^{n} \lVert f_{i} \rVert_{L^{p} (Q)} , 
\label{eq:defOfF0} \displaybreak[1] \\
F_{1} & :=
\lVert f \rVert_{
 L^{\max \left\{ 2, \frac{p(n+2)}{n+2+p} \right\}} (Q)
}
+ \sum_{i=1}^{n} \left(
 \lVert f_{i} \rVert_{L^{p} (Q)}
 + \left\lVert
  \frac{\partial f_{i}}{\partial t}
 \right\rVert_{L^{2} (Q)}
\right) , \label{eq:defOfF1}
\end{align}
and $C>0$ depends only on $n, \lambda , \Lambda , p$
and $\varepsilon$.
\end{lemma}
Now we prove our main theorem by applying
Lemma~\ref{lemma:keyestimate}.
This proof is inspired by \cite{LRU}.
\begin{proof}[of Theorem~\ref{theorem:main}]
Before going into the proof, we remark that a general constant $C$ which we used below in our estimates depends only on $n, \lambda , \Lambda , p$
and $\varepsilon_{j}$ ($j = 1, 2, 3$). To begin with the proof, let $0 < \varepsilon_{1} < \varepsilon_{2} < \varepsilon_{3}$.
Then we have
\begin{equation}\label{eq:proof11}
\sup_{\varepsilon_{2}^{2} < t \leq T}
\lVert u ( \cdot , t ) \rVert_{L^{2} ( D_{\varepsilon_{2}} )}
\leq C \left(
 \lVert u \rVert_{L^{2} (Q)} + F_{0}
\right) 
\end{equation}
and
\begin{equation}\label{eq:proof12}
\left\lVert
 \frac{\partial u}{\partial t}
\right\rVert_{L^{2} ( \widehat{Q}_{\varepsilon_{1}} )}
\leq C \left(
 \lVert u \rVert_{L^{2} (Q)} + F_{1}
\right) 
\end{equation}
by (\ref{eq:LinftyL2}) and (\ref{eq:dudtL2uL2})
in Lemma~\ref{lemma:keyestimate},
where $F_{0}$, $F_{1}$ are defined by 
(\ref{eq:defOfF0}) and (\ref{eq:defOfF1}).
On the other hand,
$u_{t} = \partial u / \partial t$
satisfies the equation
\[
\frac{\partial u_{t}}{\partial t}
- \sum_{i,j=1}^{n} \frac{\partial}{\partial x_{i}} \left(
 a_{i j} (x) \frac{\partial u_{t}}{\partial x_{j}}
\right) = \frac{\partial f}{\partial t}
- \sum_{i=1}^{n} \frac{\partial}{\partial x_{i}} \left(
 \frac{\partial f_{i}}{\partial t}
\right)
\]
by applying $\partial / \partial t$ to (\ref{eq:parabolic})
(also see Remark~\ref{remark:steklov}).
Hence we have
\begin{equation}\label{eq:proof2}
\lVert u_{t} \rVert_{L^{\infty} ( \widehat{Q}_{\varepsilon_{2}} )}
\leq C \left(
 \lVert u_{t} \rVert_{L^{2} ( \widehat{Q}_{\varepsilon_{1}} )}
 + F_{0}^{\prime}
\right)
\end{equation}
by Lemma~\ref{lemma:keyestimate} (\ref{eq:Linfty}),
where we define
\[
F_{0}^{\prime} :=
\left\lVert
 \frac{\partial f}{\partial t}
\right\rVert_{L^{\frac{p(n+2)}{n+2+p}} (Q)}
+ \sum_{i=1}^{n}
\left\lVert
 \frac{\partial f_{i}}{\partial t}
\right\rVert_{L^{p} (Q)} .
\]
In particular,
$u_t( \cdot , t ) \in L^{\infty} ( D_{\varepsilon_{2}} )$
holds for a.e.\ $t \in ( \varepsilon_{2}^{2} , T ]$.
Now we regard the equation (\ref{eq:parabolic}) as
the elliptic equation
\begin{equation}\label{eq:aselliptic}
\sum_{i,j=1}^{n} \frac{\partial}{\partial x_{i}} \left(
 a_{i j} (x) \frac{\partial u}{\partial x_{j}}
\right)
= \frac{\partial u}{\partial t}
- f + \sum_{i=1}^{n} \frac{\partial f_{i}}{\partial x_{i}}
\end{equation}
by fixing $t \in ( \varepsilon_{2}^{2} , T ]$.
We remark that
$\partial u / \partial t - f \in L^{\infty} ( D_{\varepsilon_{2}} )$.
Then, for any $\alpha^{\prime}$
with the condition (\ref{eq:alphaprime}),
we have the estimate
\begin{align}
& \sum_{m=1}^{L}
\lVert u ( \cdot , t ) \rVert_{
 C^{1, \alpha^{\prime}}
 ( \overline{D_{m}} \cap D_{\varepsilon_{3}} )
} \notag \\
& \leq C_{\sharp} \Biggl(
 \lVert u ( \cdot , t ) \rVert_{L^{2} ( D_{\varepsilon_{2}} )}
 + \left\lVert \frac{\partial u}{\partial t} ( \cdot , t )
 \right\rVert_{
  L^{\infty} ( D_{\varepsilon_{2}} )
 }
 + \lVert f ( \cdot , t ) \rVert_{L^{\infty} ( D_{\varepsilon_{2}} )}
 \notag \\ 
 & \hspace*{35ex} \mbox{}
 + \sum_{m=1}^{L} \sum_{i=1}^{n}
 \lVert f_{i} ( \cdot , t ) \rVert_{
  C^{\alpha^{\prime}} ( \overline{D_{m}})
 }
\Biggr) \label{eq:x}
\end{align}
by Theorem~\ref{theorem:LN}, where
$C_{\sharp} > 0$ depends only on
\begin{math}
n, L, \mu, \alpha , \varepsilon , \lambda , \Lambda ,
\lVert a_{i j} \rVert_{C^{\alpha^{\prime}} ( \overline{D_{m}})}
\end{math}
and the $C^{1, \alpha^{\prime}}$ norms of $D_{m}$.
Taking the supremum of the inequality 
(\ref{eq:x}) over $( \varepsilon_{2}^{2} , T ]$
with respect to $t$, 
and using (\ref{eq:proof11}), (\ref{eq:proof12}) 
and (\ref{eq:proof2}),
we have
\begin{align*}
& \sum_{m=1}^{L}
\sup_{\varepsilon_{2}^{2} < t \leq T}
\lVert u ( \cdot , t ) \rVert_{
 C^{1, \alpha^{\prime}}
 ( \overline{D_{m}} \cap D_{\varepsilon_{3}} )
} \\
& \leq C_{\sharp} \Biggl(
 \sup_{\varepsilon_{2}^{2} < t \leq T}
 \lVert u ( \cdot , t ) \rVert_{L^{2} ( D_{\varepsilon_{2}} )}
 + \left\lVert \frac{\partial u}{\partial t}
 \right\rVert_{
  L^{\infty} ( \widehat{Q}_{\varepsilon_{2}} )
 }
 + \lVert f \rVert_{L^{\infty} ( \widehat{Q}_{\varepsilon_{2}} )} \\
 & \hspace*{30ex} \mbox{}
 + \sum_{m=1}^{L} \sum_{i=1}^{n}
 \sup_{\varepsilon_{2}^{2} < t \leq T}
 \lVert f_{i} ( \cdot , t ) \rVert_{
  C^{\alpha^{\prime}} ( \overline{D_{m}})
 }
\Biggr) \displaybreak[1] \\
& \leq C_{\sharp} C \Biggl(
 \lVert u \rVert_{L^{2} (Q)}
 + F_{0} + F_{1} + F_{0}^{\prime}
 + \lVert f \rVert_{L^{\infty} ( \widehat{Q}_{\varepsilon_{2}} )} \\
 & \hspace*{30ex} \mbox{}
 + \sum_{m=1}^{L} \sum_{i=1}^{n}
 \sup_{\varepsilon_{2}^{2} < t \leq T}
 \lVert f_{i} ( \cdot , t ) \rVert_{
  C^{\alpha^{\prime}} ( \overline{D_{m}})
 }
\Biggr) ,
\end{align*}
which is the estimate we want to obtain.
\end{proof}
\begin{remark}\label{remark:steklov}
Since we assume that $u$ belongs only in $V_{2}^{1,0} (Q)$
with respect to the regularity of a weak solution, 
one may think that we cannot apply $\partial / \partial t$ directly. 
However, it is enough to consider the Steklov mean function 
and to make $h$ tend to $0$, 
where we define the Steklov mean function 
$v_{h}$ of $v$ by 
\[
v_{h} (x,t) = \frac{1}{h} \int_{t}^{t + h} v( x, \tau ) \, d \tau. 
\]
Hereafter we omit the detail with respect to this remark 
although we often apply this argument. Also see 
\cite[III \S2 p.\,141]{LSU}
and (62) in \cite[p.\,152]{LRU}, for example. 
\end{remark}
\section{Some estimates.}\label{section:someestimates}
In this section, we prove Lemma~\ref{lemma:keyestimate}.
The estimates (\ref{eq:LinftyL2}) and (\ref{eq:Linfty}) are
well-known, but we give these proofs in Appendix
for readers' convenience.
In order to show the estimate (\ref{eq:dudtL2uL2}), we prepare some necessary lemmas
for its proof.

Throughout this section, $C > 0$ denotes a general constant
depending only on $n, \lambda , \Lambda$. Also, we assume that
the coefficient $( a_{i j} )$ is a matrix-value function
defined on $D$, symmetric, positive definite,
and satisfies the condition (\ref{eq:aijxiixij}).
Moreover, we set
$Q_{r} := B_{r} ( x_{0} ) \times ( t_{0} - r^{2} , \, t_{0} ]$,
and assume that
$Q_{2 \rho} \subset D \times (0,T]$
with $0 < \rho \leq 1$.

\medskip
The following two lemmas are essentially shown in
\cite{LRU}.
We give their proofs here for the sake of completeness.

\begin{lemma}[{\cite[Lemma~3]{LRU}}]\label{lemma:DuL2}
Let $1 < r < \infty$ and $1/r + 1 / r^{\prime} = 1$.
Then a solution $u$ to {\rm (\ref{eq:parabolic})}
satisfies the estimate
\begin{equation}\label{eq:DuL2}
\lVert \nabla u \rVert_{L^{2} ( Q_{\rho} )}
\leq C \left[
 ( \rho^{n/2} + \rho^{(n+2) / r^{\prime}} ) \osc_{Q_{2 \rho}} u
 + \lVert f \rVert_{L^{r} ( Q_{2 \rho} )}
 + \sum_{i=1}^{n} \lVert f_{i} \rVert_{L^{2} ( Q_{2 \rho} )}
\right] .
\end{equation}
\end{lemma}
\begin{proof}
Let $\zeta$ be a smooth cut-off function on $Q_{2 \rho}$
satisfying $\zeta \equiv 1$ on $Q_{\rho}$,
$\zeta \equiv 0$ on
$Q_{2 \rho} \setminus Q_{3 \rho / 2}$,
$0 \leq \zeta \leq 1$ on $Q_{2 \rho}$,
and
\begin{math}
\lvert \partial \zeta / \partial t \rvert
+ \lvert \nabla \zeta \rvert^{2}
\leq C \rho^{-2}
\end{math}
on $Q_{2 \rho}$.
Let $u_{0}$ be the average value of $u$ in $Q_{2 \rho}$:
\[
u_{0} := \frac{1}{\lvert Q_{2 \rho} \rvert}
\iint_{Q_{2 \rho}} u(x,t) \, d x \, d t,
\]
where $\lvert Q_{2 \rho} \rvert$ denotes the measure of $ Q_{2 \rho}$.
Testing (\ref{eq:parabolic})
by $( u - u_{0} ) \zeta^{2}$ 
and integrating by parts
(i.e.\ taking $\varphi = ( u - u_{0} ) \zeta^{2}$
for (\ref{eq:weaksolution}). 
Also see Remark~\ref{remark:steklov}), we have
\begin{align*}
& \frac{1}{2} \int_{B_{2 \rho} ( x_{0} )}
 \bigl( ( u - u_{0} )^{2} \zeta^{2} \bigr) ( x, t_{0} ) \,
d x
- \iint_{Q_{2 \rho}}
 ( u - u_{0} )^{2} \zeta \frac{\partial \zeta}{\partial t}
d x \, d t \\
& \mbox{}
+ \iint_{Q_{2 \rho}}
 \sum_{i,j=1}^{n} a_{i j} \frac{\partial u}{\partial x_{j}}
 \frac{\partial u}{\partial x_{i}} \zeta^{2} \,
d x \, d t
+ 2 \iint_{Q_{2 \rho}}
 \sum_{i,j=1}^{n} a_{i j} \frac{\partial u}{\partial x_{j}}
 ( u - u_{0} ) \zeta \frac{\partial \zeta}{\partial x_{i}}
d x \, d t \\
& = \iint_{Q_{2 \rho}} f ( u - u_{0} ) \zeta^{2} \, d x \, d t
+ \sum_{i=1}^{n} \iint_{Q_{2 \rho}}
 \left[
  f_{i} \frac{\partial u}{\partial x_{i}} \zeta^{2}
  + 2 f_{i} ( u - u_{0} ) \zeta \frac{\partial \zeta}{\partial x_{i}}
 \right]
d x \, d t.
\end{align*}
Hence we have
\begin{align*}
& \frac{1}{2} \int_{B_{2 \rho} ( x_{0} )}
 \bigl( ( u - u_{0} )^{2} \zeta^{2} \bigr) ( x, t_{0} ) \,
d x + \lambda \iint_{Q_{2 \rho}}
 \lvert \nabla u \rvert^{2} \zeta^{2} \,
d x \, d t \\
& \leq \frac{1}{2} \int_{B_{2 \rho} ( x_{0} )}
 \bigl( ( u - u_{0} )^{2} \zeta^{2} \bigr) ( x, t_{0} ) \,
d x + \iint_{Q_{2 \rho}}
 \sum_{i,j=1}^{n} a_{i j} \frac{\partial u}{\partial x_{j}}
 \frac{\partial u}{\partial x_{i}} \zeta^{2} \,
d x \, d t \displaybreak[1] \\
& = \iint_{Q_{2 \rho}}
 ( u - u_{0} )^{2} \zeta \frac{\partial \zeta}{\partial t}
d x \, d t
- 2 \iint_{Q_{2 \rho}}
 \sum_{i,j=1}^{n} a_{i j} \frac{\partial u}{\partial x_{j}}
 ( u - u_{0} ) \zeta \frac{\partial \zeta}{\partial x_{i}}
d x \, d t \\
& \hspace*{3ex} \mbox{}
+ \iint_{Q_{2 \rho}} f ( u - u_{0} ) \zeta^{2} \, d x \, d t \\
& \hspace*{3ex} \mbox{}
+ \sum_{i=1}^{n} \iint_{Q_{2 \rho}}
 \left[
  f_{i} \frac{\partial u}{\partial x_{i}} \zeta^{2}
  + 2 f_{i} ( u - u_{0} ) \zeta \frac{\partial \zeta}{\partial x_{i}}
 \right]
d x \, d t \displaybreak[1] \\
& \leq \iint_{Q_{2 \rho}}
 ( u - u_{0} )^{2} \zeta
 \left\lvert \frac{\partial \zeta}{\partial t} \right\rvert
d x \, d t
+ \varepsilon_{1} \iint_{Q_{2 \rho}}
 \lvert \nabla u \rvert^{2} \zeta^{2} \,
d x \, d t \\
& \hspace*{3ex} \mbox{}
+ \frac{C}{\varepsilon_{1}} \iint_{Q_{2 \rho}}
 \lvert u - u_{0} \rvert^{2} \lvert \nabla \zeta \rvert^{2} \,
d x \, d t 
+ \frac{1}{2} \left(
 \iint_{Q_{2 \rho}}
  \lvert f \zeta \rvert^{r} \,
 d x \, d t
\right)^{2/r} \\
& \hspace*{3ex} \mbox{}
+ \frac{1}{2} \left(
 \iint_{Q_{2 \rho}}
  \lvert ( u - u_{0} ) \zeta \rvert^{r^{\prime}} \,
d x \, d t
\right)^{2 / r^{\prime}} 
+ \varepsilon_{1} \iint_{Q_{2 \rho}}
 \lvert \nabla u \rvert^{2} \zeta^{2} \,
d x \, d t \\
& \hspace*{3ex} \mbox{}
+ \left( \frac{1}{\varepsilon_{1}} + 1 \right)
\iint_{Q_{2 \rho}}
 \sum_{i=1}^{n} \lvert f_{i} \rvert^{2} \zeta^{2} \,
d x \, d t
+ \iint_{Q_{2 \rho}}
 \lvert u - u_{0} \rvert^{2}
 \lvert \nabla \zeta \rvert^{2} \,
d x \, d t .
\end{align*}
We now take $\varepsilon_{1} > 0$ small enough.
Then, we have
\begin{align*}
& \iint_{Q_{\rho}} \lvert \nabla u \rvert^{2} \, d x \, d t
\leq \iint_{Q_{2 \rho}}
 \lvert \nabla u \rvert^{2} \zeta^{2} \,
d x \, d t \displaybreak[1] \\
& \leq C \iint_{Q_{2 \rho}}
 ( u - u_{0} )^{2} \left[
  \zeta \left\lvert \frac{\partial \zeta}{\partial t} \right\rvert
  + \lvert \nabla \zeta \rvert^{2}
 \right]
d x \, d t \\
& \hspace*{3ex} \mbox{}
+ C \left(
 \iint_{Q_{2 \rho}}
  \lvert ( u - u_{0} ) \zeta \rvert^{r^{\prime}} \,
 d x \, d t
\right)^{2 / r^{\prime}} \\
& \hspace*{3ex} \mbox{}
+ C \left(
 \iint_{Q_{2 \rho}} \lvert f \zeta \rvert^{r} \, d x \, d t
\right)^{2/r}
+ C \iint_{Q_{2 \rho}}
 \sum_{i=1}^{n} \lvert f_{i} \rvert^{2} \zeta^{2} \,
d x \, d t \displaybreak[1] \\
& \leq C \left[
 \left( \rho^{n} + \rho^{2(n+2) / r^{\prime}} \right)
 \left( \osc_{Q_{2 \rho}} u \right)^{2}
 + \lVert f \rVert_{L^{r} ( Q_{2 \rho} )}^{2}
 + \sum_{i=1}^{n} \lVert f_{i} \rVert_{L^{2} ( Q_{2 \rho} )}^{2}
\right],
\end{align*}
because
$\lvert u(x,t) - u_{0} \rvert \leq \osc_{Q_{2 \rho}} u$
holds for any $(x,t) \in Q_{2 \rho}$. This completes the proof.
\end{proof}
\begin{lemma}[{\cite[Lemma~5]{LRU}}]\label{lemma:dudtL2}
A solution $u$ to {\rm (\ref{eq:parabolic})}
satisfies the estimate
\begin{align}
\left\lVert
 \frac{\partial u}{\partial t}
\right\rVert_{L^{2} ( Q_{\rho} )}
& \leq C \Biggl[
 \rho^{-1} \lVert \nabla u \rVert_{L^{2} ( Q_{2 \rho} )}
 + \lVert f \rVert_{L^{2} ( Q_{2 \rho} )} \notag \\
 & \hspace*{10ex}
 + \sum_{i=1}^{n}
 \left(
  \rho^{-1}
  \lVert f_{i} \rVert_{L^{2} ( Q_{2 \rho} )}
  + \left\lVert
   \frac{\partial f_{i}}{\partial t}
  \right\rVert_{L^{2} ( Q_{2 \rho} )}
 \right)
\Biggr] \label{eq:dudtL2}
\end{align}
\end{lemma}
\begin{proof}
We first take the same smooth cut-off function $\zeta$
as in the proof of Lemma~\ref{lemma:DuL2}.
Testing (\ref{eq:parabolic}) by
$( \partial u / \partial t ) \zeta^{2}$
and integrating by parts (also see Remark~\ref{remark:steklov}), 
we have
\begin{align*}
& \frac{1}{2} \int_{B_{2 \rho} ( x_{0} )}
 \sum_{i,j=1}^{n}
 \left(
  a_{i j} \frac{\partial u}{\partial x_{i}}
  \frac{\partial u}{\partial x_{j}} \zeta^{2}
 \right) ( x, t_{0} ) \,
d x \\
& \mbox{}
+ \iint_{Q_{2 \rho}}
 \left[
  \left\lvert \frac{\partial u}{\partial t} \right\rvert^{2}
  \zeta^{2}
  - \sum_{i,j=1}^{n} a_{i j} \frac{\partial u}{\partial x_{i}}
  \frac{\partial u}{\partial x_{j}}
  \zeta \frac{\partial \zeta}{\partial t}
  + 2 \sum_{i,j=1}^{n} a_{i j} \frac{\partial u}{\partial x_{j}}
  \frac{\partial u}{\partial t} \zeta
  \frac{\partial \zeta}{\partial x_{i}}
 \right]
d x \, d t \\
& = \iint_{Q_{2 \rho}}
 f \frac{\partial u}{\partial t} \zeta^{2} \,
d x \, d t 
+ \sum_{i=1}^{n} \Biggl[
 \int_{B_{2 \rho} ( x_{0} )}
  \left( f_{i} \frac{\partial u}{\partial x_{i}} \zeta^{2} \right)
  ( x, t_{0} ) \,
 d x \\
 & \hspace*{10ex} \mbox{}
 + \iint_{Q_{2 \rho}}
  \left(
   - \frac{\partial f_{i}}{\partial t}
   \frac{\partial u}{\partial x_{i}} \zeta^{2}
   - 2 f_{i} \frac{\partial u}{\partial x_{i}}
   \zeta \frac{\partial \zeta}{\partial t}
   + 2 f_{i} \frac{\partial u}{\partial t}
   \zeta \frac{\partial \zeta}{\partial x_{i}}
  \right)
 \Biggr]
d x \, d t
\end{align*}
due to
\[
\sum_{i,j=1}^{n} a_{i j}
\frac{\partial^{2} u}{\partial t \partial x_{i}}
\frac{\partial u}{\partial x_{j}} \zeta^{2}
= \frac{1}{2} \frac{\partial}{\partial t} \left(
 \sum_{i,j=1}^{n} a_{i j}
 \frac{\partial u}{\partial x_{i}}
 \frac{\partial u}{\partial x_{j}} \zeta^{2}
\right)
- \sum_{i.j=1}^{n} a_{i j}
\frac{\partial u}{\partial x_{i}} \frac{\partial u}{\partial x_{j}}
\zeta \frac{\partial \zeta}{\partial t}
\]
and
\[
f_{i} \frac{\partial^{2} u}{\partial t \partial x_{i}} \zeta^{2}
= \frac{\partial}{\partial t} \left(
 f_{i} \frac{\partial u}{\partial x_{i}} \zeta^{2}
\right)
- \frac{\partial u}{\partial x_{i}}
\frac{\partial}{\partial t} ( f_{i} \zeta^{2} ) .
\]
Hence we have
\begin{align*}
& \frac{\lambda}{2} \int_{B_{2 \rho} ( x_{0} )}
 \left( \lvert \nabla u \rvert^{2} \zeta^{2} \right)
 ( x, t_{0} ) \,
d x + \iint_{Q_{2 \rho}}
 \left\lvert \frac{\partial u}{\partial t} \right\rvert^{2}
 \zeta^{2} \,
d x \, d t \\
& \leq \frac{1}{2} \int_{B_{2 \rho} ( x_{0} )}
 \left(
  \sum_{i,j=1}^{n} a_{i j} \frac{\partial u}{\partial x_{i}}
  \frac{\partial u}{\partial x_{j}} \zeta^{2}
 \right) ( x, t_{0} ) \,
d x + \iint_{Q_{2 \rho}}
 \left\lvert \frac{\partial u}{\partial t} \right\rvert^{2}
 \zeta^{2} \,
d x \, d t \displaybreak[1] \\
& = \iint_{Q_{2 \rho}}
 \left[
  \sum_{i,j=1}^{n} a_{i j} \frac{\partial u}{\partial x_{i}}
  \frac{\partial u}{\partial x_{j}}
  \zeta \frac{\partial \zeta}{\partial t}
  - 2 \sum_{i,j=1}^{n} a_{i j} \frac{\partial u}{\partial x_{j}}
  \frac{\partial u}{\partial t} \zeta
  \frac{\partial \zeta}{\partial x_{i}}
 \right]
d x \, d t \\
& \hspace*{3ex} \mbox{}
+ \iint_{Q_{2 \rho}}
 f \frac{\partial u}{\partial t} \zeta^{2} \,
d x \, d t 
+ \sum_{i=1}^{n} \Biggl[
 \int_{B_{2 \rho} ( x_{0} )}
  \left( f_{i} \frac{\partial u}{\partial x_{i}} \zeta^{2} \right)
  ( x, t_{0} ) \,
 d x \\
 & \hspace*{10ex} \mbox{}
 + \iint_{Q_{2 \rho}}
  \left(
   - \frac{\partial f_{i}}{\partial t}
   \frac{\partial u}{\partial x_{i}} \zeta^{2}
   - 2 f_{i} \frac{\partial u}{\partial x_{i}}
   \zeta \frac{\partial \zeta}{\partial t}
   + 2 f_{i} \frac{\partial u}{\partial t}
   \zeta \frac{\partial \zeta}{\partial x_{i}}
  \right)
 \Biggr]
d x \, d t \displaybreak[1] \\
& \leq C \iint_{Q_{2 \rho}}
 \lvert \nabla u \rvert^{2} \zeta
 \left\lvert \frac{\partial \zeta}{\partial t} \right\rvert
d x \, d t
+ \varepsilon_{2} \iint_{Q_{2 \rho}}
 \left\lvert \frac{\partial u}{\partial t} \right\rvert^{2}
 \zeta^{2} \,
d x \, d t \\
& \hspace*{3ex} \mbox{}
+ \frac{C}{\varepsilon_{2}} \iint_{Q_{2 \rho}}
 \lvert \nabla u \rvert^{2}
 \lvert \nabla \zeta \rvert^{2} \,
d x \, d t 
+ \varepsilon_{2} \iint_{Q_{2 \rho}}
 \left\lvert \frac{\partial u}{\partial t} \right\rvert^{2}
 \zeta^{2} \,
d x \, d t \\
& \hspace*{3ex} \mbox{}
+ \frac{C}{\varepsilon_{2}} \iint_{Q_{2 \rho}}
 \lvert f \rvert^{2} \zeta^{2} \,
d x \, d t 
+ \varepsilon_{2} \int_{B_{2 \rho} ( x_{0} )}
 \left( \lvert \nabla u \rvert^{2} \zeta^{2} \right) ( x, t_{0} ) \,
d x \\
& \hspace*{3ex} \mbox{}
+ \frac{C}{\varepsilon_{2}} \int_{B_{2 \rho} ( x_{0} )}
 \left( \sum_{i=1}^{n} \lvert f_{i} \rvert^{2} \zeta^{2} \right)
 ( x, t_{0} ) \,
d x \\
& \hspace*{3ex} \mbox{}
+ C \iint_{Q_{2 \rho}}
 \lvert \nabla u \rvert^{2} \zeta^{2} \,
d x \, d t
+ C \iint_{Q_{2 \rho}}
 \sum_{i=1}^{n}
 \left\lvert \frac{\partial f_{i}}{\partial t} \right\rvert^{2}
 \zeta^{2} \,
d x \, d t \\
& \hspace*{3ex} \mbox{}
+ C \iint_{Q_{2 \rho}}
 \lvert \nabla u \rvert^{2} \zeta
 \left\lvert \frac{\partial \zeta}{\partial t} \right\rvert
d x \, d t
+ C \iint_{Q_{2 \rho}}
 \sum_{i=1}^{n}
 \lvert f_{i} \rvert^{2}
 \zeta \left\lvert \frac{\partial \zeta}{\partial t} \right\rvert
d x \, d t \\
& \hspace*{3ex} \mbox{}
+ \varepsilon_{2} \iint_{Q_{2 \rho}}
 \left\lvert \frac{\partial u}{\partial t} \right\rvert^{2}
 \zeta^{2} \,
d x \, d t
+ \frac{C}{\varepsilon_{2}} \iint_{Q_{2 \rho}}
 \sum_{i=1}^{n} \lvert f_{i} \rvert^{2}
 \lvert \nabla \zeta \rvert^{2} \,
d x \, d t .
\end{align*}
We remark that
\begin{align*}
\int_{B_{2 \rho} ( x_{0} )}
 ( f_{i} \zeta )^{2} ( x, t_{0} ) \,
d x
& = \int_{B_{2 \rho} ( x_{0} )}
 \int_{t_{0} - ( 2 \rho )^{2}}^{t_{0}}
  \frac{\partial}{\partial t}
  \left( ( f_{i} \zeta )^{2} \right) ( x, t ) \,
 d t \,
d x \\
& \leq C \iint_{Q_{2 \rho}}
 \left[
  \lvert f_{i} \rvert^{2} \left(
   \zeta^{2} + \zeta \left\lvert
    \frac{\partial \zeta}{\partial t}
   \right\rvert
  \right)
  + \left\lvert \frac{\partial f_{i}}{\partial t} \right\rvert^{2}
  \zeta^{2}
 \right]
d x \, d t.
\end{align*}
Therefore, by taking $\varepsilon_{2} > 0$ small enough, we have
\begin{align*}
& \iint_{Q_{\rho}}
 \left\lvert \frac{\partial u}{\partial t} \right\rvert^{2}
d x \, d t 
\leq \int_{B_{2 \rho} ( x_{0} )}
 \left( \lvert \nabla u \rvert^{2} \zeta^{2} \right)
 ( x, t_{0} ) \,
d x + \iint_{Q_{2 \rho}}
 \left\lvert \frac{\partial u}{\partial t} \right\rvert^{2}
 \zeta^{2} \,
d x \, d t \\
& \leq C \iint_{Q_{2 \rho}}
 \lvert \nabla u \rvert^{2} \left(
  \zeta^{2}
  + \zeta
  \left\lvert \frac{\partial \zeta}{\partial t} \right\rvert
  + \lvert \nabla \zeta \rvert^{2}
 \right)
d x \, d t
+ C \iint_{Q_{2 \rho}}
 \lvert f \rvert^{2} \zeta^{2} \,
d x \, d t \\
& \hspace*{3ex} \mbox{}
+ C \iint_{Q_{2 \rho}}
 \sum_{i=1}^{n} \left[
  \lvert f_{i} \rvert^{2} \left(
   \zeta^{2}
   + \zeta
   \left\lvert \frac{\partial \zeta}{\partial t} \right\rvert
   + \lvert \nabla \zeta \rvert^{2}
  \right)
  + \left\lvert \frac{\partial f_{i}}{\partial t} \right\rvert^{2}
  \zeta^{2}
 \right]
d x  \, d t \displaybreak[1] \\
& \leq C \rho^{-2} \lVert \nabla u \rVert_{L^{2} ( Q_{2 \rho} )}^{2}
+ C \lVert f \rVert_{L^{2} ( Q_{2 \rho} )}^{2}
+ C \rho^{-2} \sum_{i=1}^{n}
\lVert f_{i} \rVert_{L^{2} ( Q_{2 \rho} )}^{2} \\
& \hspace*{3ex} \mbox{}
+ C \sum_{i=1}^{n} \left\lVert
 \frac{\partial f_{i}}{\partial t}
\right\rVert_{L^{2} ( Q_{2 \rho} )}^{2} .
\hspace*{40ex}
\end{align*}
\end{proof}
We obtain the estimate (\ref{eq:dudtL2uL2})
from Lemmas~\ref{lemma:Linfty} (given in Appendix),
\ref{lemma:DuL2} and
\ref{lemma:dudtL2}.
\section{A gradient estimate of the fundamental solution.}%
\label{section:gradestoffundsol}
In this section, we consider a gradient estimate
of the fundametal solution of parabolic operators.
We first state some facts.
It is known that if coefficient $( a_{i j} )$ is
a symmetric and positive definite matrix-value 
$L^{\infty}(\mathbb{R}^n)$ function
satisfying (\ref{eq:aijxiixij}), then
there exists a fundamental solution
$\Gamma (x,t;y,s)$ of the parabolic operator
\begin{equation}\label{eq:op}
\frac{\partial}{\partial t}
- \sum_{i,j=1}^{n} \frac{\partial}{\partial x_{i}} \left(
 a_{i j} \frac{\partial}{\partial x_{j}}
\right)
\end{equation}
with the estimate
\begin{equation}\label{eq:estoffundsol}
\lvert \Gamma (x,t;y,s) \rvert
\leq \frac{C_{\ast}}{(t-s)^{n/2}} \exp \left(
 - \frac{c_{\ast} \lvert x - y \rvert^{2}}{t-s}
\right) \chi_{[ s, \infty )} (t) 
\end{equation}
for all 
$t,s \in \mathbb{R}$, 
and a.e.\ $x,y \in \mathbb{R}^{n}$, 
where $C_{\ast} , c_{\ast} > 0$ depend only on
$n, \lambda , \Lambda$
(see \cite{Aronson} or \cite{FS}, for example).
In particular, the constants $C_{\ast}$ and $c_{\ast}$
are independent of the distance between inclusions. 
If the coefficients $( a_{i j} )$ is not piecewise smooth
but H{\"o}lder continuous in the whole space $\mathbb{R}^{n}$,
then the pointwise gradient estimate
\[
\lvert \nabla_{x} \Gamma (x,t;y,s) \rvert
\leq \frac{C_{\ast}}{(t-s)^{(n+1)/2}}
\exp \left(
 - \frac{c_{\ast} \lvert x - y \rvert^{2}}{t-s}
\right) \chi_{[ s, \infty )} (t) 
\]
holds for 
$t,s \in \mathbb{R}$, a.e.\ $x,y \in \mathbb{R}^{n}$
(see \cite[Chapter IV \S11--13]{LSU}, for example).

Now, the aim of this section is to show
the gradient estimate (\ref{eq:gradientestimate})
in Theorem~\ref{theorem:gradientestimate}
even if the coefficients are piecewise $C^{\mu}$ in $D$.
We assume that $( a_{i j} )$ defined in $D$
satisfies the conditions 
(\ref{eq:aijxiixij}) and (\ref{eq:piecewiseCmu}), 
and extend it to the whole $\mathbb{R}^{n}$ 
by defining $( a_{i j} ) \equiv \Lambda I$ 
in $\mathbb{R}^{n} \setminus D$, 
where $I$ is the identity matrix. 
We remark that this extension does not destroy the conditions 
 (\ref{eq:aijxiixij}) and (\ref{eq:piecewiseCmu}). 
Then there exists a fundamental solution
$\Gamma (x,t;y,s)$ of the parabolic operator (\ref{eq:op})
with the estimate (\ref{eq:estoffundsol})
as we stated above.

\medskip
To prove our gradient estimate of the fundamental solution, we apply the following corollary from
Theorem~\ref{theorem:main}.
\begin{corollary}\label{corollary:1}
Let $0 < \rho \leq 1$.
Then a solution $u$ to the parabolic equation
\begin{equation}\label{eq:parabolicrho}
\frac{\partial u}{\partial t}
- \sum_{i,j=1}^{n} \frac{\partial}{\partial x_{i}} \left(
 a_{i j} \frac{\partial u}{\partial x_{j}}
\right) = 0 \mbox{ in }
B_{\rho} ( x_{0} ) \times ( t_{0} - \rho^{2} , \, t_{0} ]
\end{equation}
has the estimate
\begin{equation}\label{eq:estimaterho}
\lVert \nabla u \rVert_{L^{\infty} (
  B_{\rho / 2} ( x_{0} ) \times ( t_{0} - ( \rho / 2 )^{2} , \, t_{0}] )
 )
} \leq \frac{C_{\sharp}^{\prime}}{\rho^{n/2+2}}
\lVert u \rVert_{L^{2} (
  B_{\rho} ( x_{0} ) \times ( t_{0} - \rho^{2} , \, t_{0} ]
 )
} ,
\end{equation}
where $C_{\sharp}^{\prime} > 0$ depends only on
$n, L, \mu, \alpha , \lambda , \Lambda$,
and
\begin{math}
\lVert a_{i j} \rVert_{C^{\alpha^{\prime}} ( \overline{D_{m}})}
\end{math}
and the $C^{1, \alpha^{\prime}}$ norms of $D_{m}$
for some $\alpha^{\prime}$ with {\rm (\ref{eq:alphaprime})}.
\end{corollary}
\begin{proof}
It is enough to apply the scaling argument.
To begin with, let
$\rho y = x - x_{0}$,
$\rho^{2} ( s - 1 ) = t - t_{0}$
and
\begin{align}
& \widetilde{u} (y,s) := u(x,t)
= u \bigl( \rho y + x_{0} , \, \rho^{2} (s-1) + t_{0} \bigr) , 
\label{eq:tilde} \displaybreak[1] \\ 
& \widetilde{a}_{i j} (y)
:= a_{i j} (x) = a_{i j} ( \rho y + x_{0} ), \notag 
\displaybreak[1] \\
& \widetilde{D}_{m}
:= \left\{
 \frac{1}{\rho} ( x - x_{0} ) : x \in D_{m}
\right\} . \notag
\end{align}
Then we have
\begin{equation}\label{eq:tildeu}
\frac{\partial \widetilde{u}}{\partial s}
- \sum_{i,j=1}^{n} \frac{\partial}{\partial y_{i}} \left(
 \widetilde{a}_{i j} \frac{\partial \widetilde{u}}{\partial y_{j}}
\right) = 0 \mbox{ in } B_{1} (0) \times (0,1].
\end{equation}
Therefore, by noting Remark~\ref{remark:indepofrho}, we have 
\[
\lVert \nabla \widetilde{u} \rVert_{
 L^{\infty} ( B_{1/2} (0) \times ( 3/4, 1 ] )
} \leq C_{\sharp}^{\prime}
\lVert \widetilde{u} \rVert_{L^{2} ( B_{1} (0) \times (0,1) )}
\]
by Theorem~\ref{theorem:main},
where $C_{\sharp}^{\prime}$ depends only on
\begin{math}
n, L, \mu, \alpha , \lambda , \Lambda ,
\lVert a_{i j} \rVert_{C^{\alpha^{\prime}} ( \overline{D_{m}})}
\end{math},
and the $C^{1, \alpha^{\prime}}$ seminorms of $D_{m}$.
By this estimate and the definition (\ref{eq:tilde}),
we obtain the estimate (\ref{eq:estimaterho}).
\end{proof}
\begin{remark}\label{remark:indepofrho}
One may think that a constant $C_{\sharp}^{\prime}$
depends also on $\rho$
since 
\begin{math}
\lVert \widetilde{a}_{i j} \rVert_{
 C^{\alpha^{\prime}} ( \overline{\widetilde{D}_{m}} )
}
\end{math} 
and the $C^{1, \alpha^{\prime}}$ norms of 
$\widetilde{D}_{m}$
depend on $\rho$. 
However, we can take $C_{\sharp}^{\prime}$
independent of $\rho$
by taking the following into consideration. 

First we consider 
\begin{align*}
\lVert \widetilde{a}_{i j} \rVert_{
 C^{\alpha^{\prime}} ( \overline{\widetilde{D}_{m}} )
}
& = \lVert \widetilde{a}_{i j} \rVert_{
 C^{0} ( \overline{\widetilde{D}_{m}} )
} + [ \widetilde{a}_{i j} ]_{
 C^{\alpha^{\prime}} ( \overline{\widetilde{D}_{m}} )
} \\
& := \sup_{y \in \overline{\widetilde{D}_{m}}}
\lvert \widetilde{a}_{i j} (y) \rvert
+ \sup_{y, \eta \in \overline{\widetilde{D}_{m}}} 
\frac{
 \lvert \widetilde{a}_{i j} (y) - \widetilde{a}_{i j} ( \eta ) \rvert
}{\lvert y - \eta \rvert^{\alpha^{\prime}}} . 
\end{align*}
It is easy to show 
\[
\lVert \widetilde{a}_{i j} \rVert_{
 C^{0} ( \overline{\widetilde{D}_{m}} )
} = \lVert a_{i j} \rVert_{C^{0} ( \overline{D_{m}} )}
\]
and
\[
[ \widetilde{a}_{i j} ]_{
 C^{\alpha^{\prime}} ( \overline{\widetilde{D}_{m}} )
} = \rho^{\alpha^{\prime}} [ a_{i j} ]_{
 C^{\alpha^{\prime}} ( \overline{D_{m}} )
} \leq [ a_{i j} ]_{C^{\alpha^{\prime}} ( \overline{D_{m}} )} . 
\]
Then we have 
\[
\lVert \widetilde{a}_{i j} \rVert_{
 C^{\alpha^{\prime}} ( \overline{\widetilde{D}_{m}} )
}
\leq \lVert a_{i j} \rVert_{C^{\alpha^{\prime}} ( \overline{D_{m}} )} . 
\]

Next we consider the $C^{1, \alpha^{\prime}}$ norms of 
$\widetilde{D}_{m}$. 
We need to recall the proofs of the results of 
\cite{LiNirenberg} and \cite{LiVogelius}
more carefully. 
In the case when we consider 
the $L^{\infty}$-norm of $\nabla \widetilde{u}$ 
for a solution $\widetilde{u}$ to the equation (\ref{eq:tildeu}), 
the influence of the $C^{1, \alpha^{\prime}}$ norms of 
subdomains $\widetilde{D}_{m}$ appears only 
in the following constant $C$ in (\ref{eq:LiVogeliusp118C}): 
We estimate 
$O \bigl( \lvert x^{\prime} \rvert^{1 + \alpha} \bigr)$
in the equation (49) in \cite[p.\ 118]{LiVogelius}, i.e.\ 
\begin{equation}\tag{49}
f_{m} ( x^{\prime} )
= f_{m} ( 0^{\prime} ) 
+ \nabla f_{m} ( 0^{\prime} ) x^{\prime}
+ O \bigl( \lvert x^{\prime} \rvert^{1 + \alpha} \bigr)
\end{equation}
as 
\begin{equation}\label{eq:LiVogeliusp118C}
\bigl\lvert 
 O \bigl( \lvert x^{\prime} \rvert^{1 + \alpha} \bigr)
\bigr\rvert
\leq C \lvert x \rvert^{1 + \alpha} 
\end{equation}
(See also \cite[Lemma~4.3]{LiNirenberg}).
Here $C^{1, \alpha}$ functions $f_{m}$ 
are defined in the cube $(-1,1)^{n}$, 
and the graphs of $f_{m}$ describe $\partial D_{m}$. 
Now we remark that the constant $C$ in (\ref{eq:LiVogeliusp118C})
depends only on the $C^{1, \alpha}$ seminorms of $f_{m}$. 
We consider the variable change $\rho y = x$. 
Then the graph $x_{n} = f_{m} ( x^{\prime} )$
is changed to $y_{n} = \widetilde{f}_{m} ( y^{\prime} )$, 
where 
\begin{math}
\widetilde{f}_{m} ( y^{\prime} ) 
:= \rho^{-1} f_{m} ( \rho y^{\prime} )
\end{math}, and we have 
\begin{align*}
[ \widetilde{f}_{m} ]_{C^{1, \alpha} ( (-1,1)^{n} )}
& \leq [ \widetilde{f}_{m} ]_{
 C^{1, \alpha} ( ( - 1 / \rho , 1 / \rho )^{n} )
} \\
& = \rho^{\alpha} [ f_{m} ]_{C^{1, \alpha} ( (-1,1)^{n} )} 
\leq [ f_{m} ]_{C^{1, \alpha} ( (-1,1)^{n} )} . 
\end{align*}
Hence, even when we consider the variable change $\rho y = x$, 
we can take the constant $C$ in (\ref{eq:LiVogeliusp118C})
independent of $\rho$. 

Considering the circumstances mentioned above, 
we can take $C_{\sharp}^{\prime} > 0$
independent of $\rho$. 
\end{remark}

Now we state the estimate of
$\nabla_{x} \Gamma (x,t;y,s)$.
\begin{theorem}\label{theorem:gradientestimate}
We have
\begin{equation}\label{eq:gradientestimate}
\lvert \nabla_{x} \Gamma (x,t;y,s) \rvert
\leq \frac{C}{( t - s )^{(n+1)/2}}
\exp \left(
 - \frac{c \lvert x - y \rvert^{2}}{t-s}
\right)
\end{equation}
for a.e.\ $x, y \in \mathbb{R}^{n}$ and $t > s$
with $\lvert x - y \rvert^{2} + t - s \leq 16$,
where
$C, c>0$ depend only on
\begin{math}
n, L, \mu, \alpha , \lambda , \Lambda ,
\lVert a_{i j} \rVert_{C^{\alpha^{\prime}} ( \overline{D_{m}})}
\end{math}
and the $C^{1, \alpha^{\prime}}$ seminorms of $D_{m}$
for some $\alpha^{\prime}$ with {\rm (\ref{eq:alphaprime})}. 
\end{theorem}
We prove Theorem~\ref{theorem:gradientestimate}
in the same way as the proof of
\cite[Proposition 3.6]{DiCristoVessella}.
We first show the following lemmas.
\begin{lemma}\label{lemma:GammaL2a}
Let 
$\rho := ( \lvert x_{0} - \xi \rvert^{2} + t_{0} - \tau )^{1/2} / 4$. 
Then
\[
\int_{t_{0} - \rho^{2}}^{t_{0}} \int_{B_{\rho} ( x_{0} )}
 \lvert \Gamma( x, t; \xi , \tau ) \rvert^{2} \,
d x \, d t
\leq \frac{( C_{\ast}^{\prime} )^{2} \rho^{n}}{( t_{0} - \tau )^{n-1}}
\exp \left(
 - \frac{2 c_{\ast}^{\prime} \lvert x_{0} - \xi \rvert^{2}}
        {t_{0} - \tau}
\right)
\]
for $t_{0} > \tau$, where
$C_{\ast}^{\prime} , c_{\ast}^{\prime} > 0$
depend only on $n, \lambda , \Lambda$.
\end{lemma}
\begin{proof}
By (\ref{eq:estoffundsol}), it is enough to obtain the estimate
\begin{align}
I_{0} 
& :=  \int_{t_{0} - \rho^{2}}^{t_{0}} \int_{B_{\rho} ( x_{0} )}
 \frac{1}{( t - \tau )^{n}}
 \exp \left(
  - \frac{2 c_{\ast} \lvert x - \xi \rvert^{2}}{t - \tau}
 \right) \chi_{[ \tau , \infty )} (t) \,
d x \, d t \notag \\
& \leq 
\frac{( C_{\ast}^{\prime} )^{2} \rho^{n}}{( t_{0} - \tau )^{n-1}}
\exp \left(
 - \frac{2 c_{\ast}^{\prime} \lvert x_{0} - \xi \rvert^{2}}
        {t_{0} - \tau}
\right). \label{eq:I0}
\end{align}
We consider the following three cases:
\[
{\rm (i)}~ t_{0} - \rho^{2} \leq \tau < t_{0} , \quad
{\rm (ii)}~ t_{0} - 2 \rho^{2} \leq \tau \leq t_{0} - \rho^{2} ,
\quad
{\rm (iii)}~ \tau \leq t_{0} - 2 \rho^{2} .
\]
Now we consider the case (i).
Then we have
$( \sqrt{15} - 1 ) \rho \leq \lvert x - \xi \rvert$
for any $x \in B_{\rho} ( x_{0} )$, because
$\lvert x_{0} - \xi \rvert \geq \sqrt{15} \, \rho$.
Hence we have
\[
I_{0} \leq
\int_{\tau}^{t_{0}} \int_{B_{\rho} ( x_{0} )}
 \frac{1}{( t - \tau )^{n}}
 \exp \left(
  - \frac{c_{1} \rho^{2}}{t - \tau}
 \right)
d x \, d t
= \lvert B_{1} (0) \rvert \rho^{n}
\int_{0}^{t_{0} - \tau} \varphi_{1} (s) \, d s,
\]
where $\varphi_{1} (s) := s^{-n} \exp ( - c_{1} \rho^{2} / s )$ and
$c_{1} := 2 ( \sqrt{15} - 1 )^{2} c_{\ast}$.
If $0 < t_{0} - \tau \leq c_{1} \rho^{2} / n$,
then we have
\[
\int_{0}^{t_{0} - \tau} \varphi_{1} (s) \, d s
\leq \int_{0}^{t_{0} - \tau} \varphi_{1} ( t_{0} - \tau ) \, d s
= ( t_{0} - \tau )^{-n+1} \exp \left(
 - \frac{c_{1} \rho^{2}}{t_{0} - \tau},
\right)
\]
because $\varphi_{1} (s) \leq \varphi_{1} ( t_{0} - \tau )$ holds for any
$s \in [ 0, \, t_{0} - \tau ]$.
On the other hand, if
$c_{1} \rho^{2} / n \leq t_{0} - \tau \leq \rho^{2}$,
then we have
\begin{align*}
\int_{0}^{t_{0} - \tau} \varphi_{1} (s) \, d s
& \leq \int_{0}^{t_{0} - \tau}
 \varphi_{1} \left( \frac{c_{1} \rho^{2}}{n} \right)
d s
= \left( \frac{n}{c_{1}} \right)^{n} ( t_{0} - \tau )
\rho^{- 2 n} \exp ( - n ) \\
& \leq \left( \frac{n}{c_{1}} \right)^{n} ( t_{0} - \tau )^{1-n}
\exp \left(
 - \frac{c_{1} \rho^{2}}{t_{0} - \tau}
\right),
\end{align*}
where we used the properties that 
\begin{align*}
& \varphi_{1} (s) 
\leq \varphi_{1} \left( \frac{c_{1} \rho^{2}}{n} \right)
\mbox{ for any } 0 < s \leq t_{0} - \tau ; 
\displaybreak[1] \\
& n \geq \frac{c_{1} \rho^{2}}{t_{0} - \tau} \mbox{, and }
\rho^{2} \geq t_{0} - \tau . 
\end{align*}
Summing up, we have
\[
I_{0} \leq \max \left\{
 1, \, \left( \frac{n}{c_{1}} \right)^{n}
\right\} \lvert B_{1} (0) \rvert
\rho^{n} ( t_{0} - \tau )^{1-n} \exp \left(
 - \frac{c_{1} \rho^{2}}{t_{0} - \tau}
\right) .
\]

Let us consider the case (ii).
Then we have $( \sqrt{14} - 1 ) \rho \leq \lvert x - \xi \rvert$
for all $x \in B_{\rho} ( x_{0} )$, 
because $\lvert x_{0} - \xi \rvert \geq \sqrt{14} \, \rho$.
Hence we have
\begin{align*}
I_{0} 
& \leq \int_{t_{0} - \rho^{2}}^{t_{0}} \int_{B_{\rho} ( x_{0} )}
 \frac{1}{( t - \tau )^{n}}
 \exp \left(
  - \frac{c_{2} \rho^{2}}{t - \tau}
 \right)
d x \, d t \\
& = \lvert B_{1} (0) \rvert \rho^{n}
\int_{t_{0} - \rho^{2} - \tau}^{t_{0} - \tau}
 \varphi_{2} (s) \,
d s,
\end{align*}
where $\varphi_{2} (s) := s^{-n} \exp ( - c_{2} \rho^{2} / s )$ and
$c_{2} := 2 ( \sqrt{14} - 1 )^{2} c_{\ast}$.
In a similary way as the case (i),
if $\rho^{2} \leq t_{0} - \tau \leq c_{2} \rho^{2} / n$,
then we have
\begin{align*}
\int_{t_{0} - \rho^{2} - \tau}^{t_{0} - \tau}
 \varphi_{2} (s) \,
d s
& \leq \int_{t_{0} - \rho^{2} - \tau}^{t_{0} - \tau}
 \varphi_{2} ( t_{0} - \tau ) \,
d s
= \rho^{2} ( t_{0} - \tau )^{-n}
\exp \left( - \frac{c_{2} \rho^{2}}{t_{0} - \tau} \right) \\
& \leq ( t_{0} - \tau )^{-n+1}
\exp \left( - \frac{c_{2} \rho^{2}}{t_{0} - \tau} \right),
\end{align*}
because $\varphi_{2} (s) \leq \varphi ( t_{0} - \tau )$
for any $s \in [ t_{0} - \rho^{2} - \tau , \, t_{0} - \tau ]$,
and we have $\rho^{2} \leq t_{0} - \tau$.
On the other hand,
if $c_{2} \rho^{2} / n \leq t_{0} - \tau \leq 2 \rho^{2}$,
then we have
\begin{align*}
\int_{t_{0} - \rho^{2} - \tau}^{t_{0} - \tau}
 \varphi_{2} (s) \,
d s
& \leq \int_{t_{0} - \rho^{2} - \tau}^{t_{0} - \tau}
 \varphi_{2} \left( \frac{c_{2} \rho^{2}}{n} \right)
d s
= \left( \frac{n}{c_{2}} \right)^{n}
\rho^{-2n+2} \exp ( -n ) \\
& \leq 2^{n-1} \left( \frac{n}{c_{2}} \right)^{n}
( t_{0} - \tau )^{1-n}
\exp \left( - \frac{c_{2} \rho^{2}}{t_{0} - \tau} \right),
\end{align*}
where we used the properties that 
\begin{align*}
& \varphi_{2} (s) 
\leq \varphi_{2} \left( \frac{c_{2} \rho^{2}}{n} \right)
\mbox{ for any } t_{0} - \rho^{2} - \tau \leq s \leq t_{0} - \tau ; 
\displaybreak[1] \\
& n \geq \frac{c_{2} \rho^{2}}{t_{0} - \tau} \mbox{, and }
\rho^{2} \geq \frac{t_{0} - \tau}{2} . 
\end{align*}
Summing up, we have
\[
I_{0} \leq \lvert B_{1} (0) \rvert
\max \left\{ 1, \, 2^{n-1} \left( \frac{n}{c_{2}} \right)^{n} \right\}
\rho^{n} ( t_{0} - \tau )^{1-n}
\exp \left( - \frac{c_{2} \rho^{2}}{t_{0} - \tau} \right) .
\]
Finally we consider the case (iii).
We first remark that
\[
\int_{t_{0} - \rho^{2}}^{t_{0}} ( t - \tau )^{-n} \, d t
\leq 
\left\{
 \begin{aligned}
 & \frac{1}{n-1} ( t_{0} - \rho^{2} - \tau )^{-n+1}
 & & \mbox{if } n \geq 2, \\
 & \log 2 & & \mbox{if } n = 1,
 \end{aligned}
\right.
\]
because $t_{0} - \tau \leq 2 ( t_{0} - \rho^{2} - \tau )$.
In particular, we have
\[
\int_{t_{0} - \rho^{2}}^{t_{0}} ( t - \tau )^{-n} \, d t
\leq ( t_{0} - \rho^{2} - \tau )^{-n+1}
\leq 2^{n-1} ( t_{0} - \tau )^{-n+1} .
\]
Hence we have
\begin{align*}
I_{0}
& \leq \lvert B_{1} (0) \rvert \rho^{n}
\int_{t_{0} - \rho^{2}}^{t_{0}} ( t - \tau )^{-n} \, d t
\leq 2^{n-1} \lvert B_{1} (0) \rvert \rho^{n}
( t_{0} -  \tau )^{-n+1} \\
& \leq 2^{n-1} \exp (8) \lvert B_{1} (0) \rvert \rho^{n}
( t_{0} - \tau )^{-n+1}
\exp \left(
 - \frac{\lvert x_{0} - \xi \rvert^{2}}{t_{0} - \tau}
\right),
\end{align*}
because
\begin{math}
\lvert x_{0} - \xi \rvert^{2} / ( t_{0} - \tau)
\leq ( 4 \rho )^{2} / 2 \rho^{2} = 8
\end{math}.

Therefore we have the estimate (\ref{eq:I0}) in every case.
\end{proof}
Now we prove Theorem~\ref{theorem:gradientestimate}.
\begin{proof}[of Theorem~\ref{theorem:gradientestimate}]
Let $x_{0} , \xi \in \mathbb{R}^{n}$ and $t_{0} > \tau$.
Let 
\begin{math}
\rho := ( \lvert x_{0} - \xi \rvert^{2} + t_{0} - \tau )^{1/2} / 4
\leq 1 
\end{math}.
Then, by Corollary~\ref{corollary:1},
we have
\begin{align*}
& \lVert \nabla_{x} \Gamma ( \cdot , \cdot ; \xi , \tau ) \rVert_{
 L^{\infty} (
  B_{\rho / 2} ( x_{0} ) \times ( t_{0} - ( \rho / 2 )^{2} , \, t_{0} )
 )
} \\
& \leq \frac{C_{\sharp}^{\prime}}{\rho^{n/2 + 2}}
\lVert
 \Gamma ( \cdot , \cdot ; \xi , \tau )
\rVert_{L^{2}
 ( B_{\rho} ( x_{0} ) \times ( t_{0} - \rho^{2} , \, t_{0} ) )
},
\end{align*}
because we have
\[
\frac{\partial \Gamma}{\partial t} ( x, t; \xi , \tau )
- \sum_{i,j=1}^{n} \frac{\partial}{\partial x_{i}} \left(
 a_{i j} (x) \frac{\partial \Gamma}{\partial x_{j}}
 ( x, t; \xi , \tau )
\right) = 0 \mbox{ in }
B_{\rho} ( x_{0} ) \times ( t_{0} - \rho^{2} , t_{0} ] .
\]
By this estimate and Lemma~\ref{lemma:GammaL2a}, we have
\begin{align*}
& \lVert \nabla_{x} \Gamma ( \cdot , \cdot ; \xi , \tau ) \rVert_{
 L^{\infty} (
  B_{\rho / 2} ( x_{0} ) \times ( t_{0} - ( \rho / 2 )^{2} , \, t_{0} ]
 )
} \\
& \leq \frac{C_{\sharp}^{\prime}}{\rho^{n/2 + 2}}
\lVert
 \Gamma ( \cdot , \cdot ; \xi , \tau )
\rVert_{L^{2}
 ( B_{\rho} ( x_{0} ) \times ( t_{0} - \rho^{2} , \, t_{0} ] )
} \displaybreak[1] \\
& \leq \frac{C_{\sharp}^{\prime} C_{\ast}^{\prime}}{\rho^{2}}
\frac{1}{( t_{0} - \tau )^{(n-1)/2}}
\exp \left(
 - \frac{c_{\ast}^{\prime} \lvert x_{0} - \xi \rvert^{2}}
        {t_{0} - \tau}
\right) \displaybreak[1] \\
& \leq
\frac{16 C_{\sharp}^{\prime} C_{\ast}^{\prime}}
     {( t_{0} - \tau )^{(n+1)/2}}
\exp \left(
 - \frac{c_{\ast}^{\prime} \lvert x_{0} - \xi \rvert^{2}}
        {t_{0} - \tau}
\right),
\end{align*}
because we have
$\rho^{-2} \leq 16 ( t_{0} - \tau )^{-1}$.
Hence the proof is completed. 
\end{proof}
\renewcommand{\thesection}{Appendix}
\renewcommand{\theequation}{A.\arabic{equation}}
\renewcommand{\thetheorem}{A.\arabic{theorem}}
\section{}
In Appendix, we show the estimates
(\ref{eq:LinftyL2}) and (\ref{eq:Linfty})
in Lemma~\ref{lemma:keyestimate}
for the sake of completeness. 
To begin with,
we give some embedding lemma
which is necessary to show the estimates
(\ref{eq:LinftyL2}) and (\ref{eq:Linfty}).
First, the following Gagliardo-Nirenberg's inequality is well-known
(see \cite[p.\ 24, Theorem~9.3]{Friedman}, for example).
\begin{lemma}[Gagliardo-Nirenberg's inequality]%
\label{lemma:GagliardoNirenberg}%
Let
$r, s$
be any numbers satisfying
$1 \leq r, s \leq \infty$,
and let
$j, k$
be any integers satisfying
$0 \leq j < k$.
If $u$ is any function in
$W_{s}^{k} ( \mathbb{R}^{n} ) \cap L^{r} ( \mathbb{R}^{n} )$,
then
\begin{equation}\label{eq:GagliardoNirenberg}
\lVert D^{j} u \rVert_{L^{q}( \mathbb{R}^{n} )}
\leq C_{1} \lVert D^{k} u \rVert_{L^{s}( \mathbb{R}^{n} )}^{\gamma}
\lVert u \rVert_{L^{r}( \mathbb{R}^{n} )}^{1 - \gamma} ,
\end{equation}
where
\begin{equation}\label{eq:gamma}
\frac{1}{q} = \frac{j}{n} +
\gamma \left( \frac{1}{s} - \frac{k}{n} \right)
+ \frac{1 - \gamma}{r}
\end{equation}
for all $\gamma$ in the interval
\[
\frac{j}{k} \leq \gamma \leq 1,
\]
where a positive constant $C_{1}$
depends only on $n, k, j, r, s, \gamma$,
with the following exception:
If $k - j - n/s$ is a nonnegative integer,
then {\rm (\ref{eq:GagliardoNirenberg})} holds
only for $j/k \leq \gamma < 1$.
\end{lemma}

\medskip
Then, as an application of Lemma~\ref{lemma:GagliardoNirenberg}, we have the following embedding lemma.
\begin{lemma}[embedding lemma]\label{lemma:SobolevV2}
Let $H_{0}^{1} (D)$ be the usual $L^2$-Sobolev space with supports in $\overline D$ and 
\begin{math}
v \in L^{\infty} \bigl( 0, T; L^{2} (D) \bigr)
\cap L^{2} \bigl( 0, T; H_{0}^{1} (D) \bigr)
\end{math}.
Then $v \in L^{2(n+2)/n} (Q)$ holds.
Moreover, we have the estimate
\begin{align}
\lVert v \rVert_{L^{2(n+2)/n} (Q)}
& \leq C_{1}
\lVert v \rVert_{L^{\infty} ( 0, T; L^{2} (D) )}^{2/(n+2)}
\lVert \nabla v \rVert_{L^{2} (Q)}^{n/(n+2)} \notag \\
& \leq C_{1} \left(
 \lVert v \rVert_{L^{\infty} ( 0, T; L^{2} (D) )}
 + \lVert \nabla v \rVert_{L^{2} (Q)}
\right) , \label{eq:V2}
\end{align}
where a positive constant $C_{1}$ depends only on $n$,
and we denote $Q := D \times (0,T]$.
\end{lemma}
\begin{proof}
We apply Lemma~\ref{lemma:GagliardoNirenberg}
with $q = 2(n+2)/n$, $r=2$, $s=2$, $k=1$ and $j=0$.
Then the equation (\ref{eq:gamma}) yields $\gamma = n/(n+2)$.
Hence we have
\[
\lVert v ( \cdot , t ) \rVert_{L^{2(n+2)/n} (D)}
\leq C_{1} \lVert \nabla v ( \cdot , t ) \rVert_{L^{2} (D)}^{n/(n+2)}
\lVert v ( \cdot , t ) \rVert_{L^{2} (D)}^{2/(n+2)} .
\]
Therefore we have
\begin{align*}
\lVert v \rVert_{L^{2(n+2)/n} (Q)}^{2(n+2)/n}
& = \int_{0}^{T}
 \lVert v ( \cdot , t ) \rVert_{L^{2(n+2)/n} (D)}^{2(n+2)/n} \,
d t \\
& \leq \int_{0}^{T}
 \left(
  C_{1} \lVert \nabla v ( \cdot , t ) \rVert_{L^{2} (D)}^{n/(n+2)}
  \lVert v ( \cdot , t ) \rVert_{L^{2} (D)}^{2/(n+2)}
 \right)^{2(n+2)/n} \,
d t \displaybreak[1] \\
& \leq C_{1}^{2(n+2)/n}
\lVert v \rVert_{L^{\infty} ( 0, T; L^{2} (D) )}^{4/n}
\lVert \nabla v \rVert_{L^{2} (Q)}^{2} .
\end{align*}
By this inequality and Young's inequality,
we have the estimate (\ref{eq:V2}).
\end{proof}
Based on Di~Giorgi's famous argument, we start to estimate solutions to the parabolic equation
(\ref{eq:parabolic}).
By testing $\max\{ u - k,0\}\zeta^{2}$ to (\ref{eq:parabolic}) we have the following lemma.
\begin{lemma}\label{lemma:DG}
Let $p > 2$. 
Let
\begin{math}
Q_{\rho} := B_{\rho} ( x_{0} ) \times ( t_{0} - \rho^{2} , t_{0} ]
\subset Q
\end{math}
and
\begin{math}
\zeta \in C^{\infty} \bigl(
 [ t_{0} - \rho^{2} , t_{0} ] ; C_{0}^{\infty} ( B_{\rho} ( x_{0} ) )
\bigr)
\end{math} 
satisfy
$0 \leq \zeta \leq 1$
and 
$\zeta ( \cdot , t_{0} - \rho^{2} ) = 0$.
Then a solution $u$ to the parabolic equation
{\rm (\ref{eq:parabolic})} satisfies
\begin{align}
& \lVert ( u - k )_{+} \zeta \rVert_{
 L^{\infty} ( t_{0} - \rho^{2} , t_{0} ; L^{2} ( B_{\rho} ( x_{0} ) ) )
}^{2}
+ \bigl\lVert
 \nabla \bigl( ( u - k )_{+} \zeta \bigr)
\bigr\rVert_{L^{2} ( Q_{\rho} )}^{2} \notag \\
& \leq C_{2} \Biggl[
 \left(
  \left\lVert
   \frac{\partial \zeta}{\partial t}
  \right\rVert_{L^{\infty} ( Q_{\rho} )}
  + \lVert \nabla \zeta \rVert_{L^{\infty} ( Q_{\rho} )}^{2}
 \right) \lVert ( u - k )_{+} \rVert_{L^{2} ( Q_{\rho} )}^{2}
 \notag \\
 & \hspace*{30ex} \mbox{}
 + F_{0, \rho}^{2}
 \bigl\lvert
  Q_{\rho}  \cap \{ u(x,t) > k \}
 \bigr\rvert^{1 - 2 / p}
\Biggr] \label{eq:DG}
\end{align}
for any $k \in \mathbb{R}$,
where $v_{+} (x) := \max \{ v(x), 0 \}$,
\begin{equation}\label{eq:F0rho}
F_{0, r} := \lVert f \rVert_{L^{\frac{p(n+2)}{n+2+p}} ( Q_{r} )}
+ \sum_{i=1}^{n} \lVert f_{i} \rVert_{L^{p} ( Q_{r} )}
\mbox{ for } r>0
\end{equation}
and $C_{2} > 0$ depends only on $n, \Lambda$ and $\lambda$.
\end{lemma}
\begin{proof}
Multiplying (\ref{eq:parabolic}) by
$( u - k )_{+} \zeta^{2}$
and integrating it over
\begin{math}
Q_{\rho}^{\prime}
:= B_{\rho} ( x_{0} ) \times ( t_{0} - \rho^{2} , t^{\prime} )
\end{math}
(also see Remark~\ref{remark:steklov}),
we have
\begin{align*}
\mbox{(LHS)}
& = \iint_{Q_{\rho}^{\prime}}
 \left( \frac{\partial}{\partial t} ( u - k )_{+} \right)
 ( u - k )_{+} \zeta^{2} \,
d x \, d t \\
& \hspace*{3ex} \mbox{}
- \sum_{i,j=1}^{n}
\iint_{Q_{\rho}^{\prime}}
 \frac{\partial}{\partial x_{i}} \left(
  a_{i j} \frac{\partial}{\partial x_{j}} ( u - k )_{+}
 \right)
 ( u - k )_{+} \zeta^{2} \,
d x \, d t \displaybreak[1] \\
& = \frac{1}{2}
\iint_{Q_{\rho}^{\prime}}
 \left( \frac{\partial}{\partial t} ( u - k )_{+}^{2} \right)
 \zeta^{2} \,
d x \, d t \\
& \hspace*{3ex} \mbox{}
+ \sum_{i,j=1}^{n}
\iint_{Q_{\rho}^{\prime}}
 a_{i j} \frac{\partial}{\partial x_{j}} ( u - k )_{+}
 \frac{\partial}{\partial x_{i}} \bigl( ( u - k )_{+} \zeta^{2} \bigr)
 \,
d x \, d t \displaybreak[1] \\
& = \frac{1}{2}
\iint_{Q_{\rho}^{\prime}}
 \left[
  \frac{\partial}{\partial t} \bigl(
   ( u - k )_{+}^{2} \zeta^{2}
  \bigr)
  - 2 ( u - k )_{+}^{2} \zeta \frac{\partial \zeta}{\partial t}
 \right]
d x \, d t \\
& \hspace*{5ex} \mbox{}
+ \sum_{i,j=1}^{n}
\iint_{Q_{\rho}^{\prime}}
 a_{i j}
 \frac{\partial}{\partial x_{j}} \bigl( ( u - k )_{+} \zeta \bigr)
 \frac{\partial}{\partial x_{i}} \bigl( ( u - k )_{+} \zeta \bigr) \,
d x \, d t \\
& \hspace*{5ex} \mbox{}
- \sum_{i,j=1}^{n}
\iint_{Q_{\rho}^{\prime}}
 a_{i j} ( u - k )_{+}^{2} \frac{\partial \zeta}{\partial x_{j}}
 \frac{\partial \zeta}{\partial x_{i}} \,
d x \, d t \displaybreak[1] \\
& =
\frac{1}{2} 
\int_{B_{\rho} ( x_{0} )}
 ( u - k )_{+}^{2} \zeta^{2} \,
d x \biggr|_{t = t^{\prime}} 
- \iint_{Q_{\rho}^{\prime}}
 ( u - k )_{+}^{2} \zeta \frac{\partial \zeta}{\partial t} \,
d x \, d t \\
& \hspace*{5ex} \mbox{}
+ \sum_{i,j=1}^{n}
\iint_{Q_{\rho}^{\prime}}
 a_{i j}
 \frac{\partial}{\partial x_{j}} \bigl( ( u - k )_{+} \zeta \bigr)
 \frac{\partial}{\partial x_{i}} \bigl( ( u - k )_{+} \zeta \bigr) \,
d x \, d t \\
& \hspace*{5ex} \mbox{}
- \sum_{i,j=1}^{n}
\iint_{Q_{\rho}^{\prime}}
 a_{i j} ( u - k )_{+}^{2} \frac{\partial \zeta}{\partial x_{j}}
 \frac{\partial \zeta}{\partial x_{i}} \,
d x \, d t .
\end{align*}
Hence we have
\begin{align}
& \frac{1}{2} \int_{B_{\rho} ( x_{0} )}
 ( u - k )_{+}^{2} \zeta^{2} \,
d x \biggr|_{t = t^{\prime}} \notag \\
& \mbox{}
+ \sum_{i,j=1}^{n}
\iint_{Q_{\rho}^{\prime}}
 a_{i j}
 \frac{\partial}{\partial x_{j}} \bigl( ( u - k )_{+} \zeta \bigr)
 \frac{\partial}{\partial x_{i}} \bigl( ( u - k )_{+} \zeta \bigr) \,
d x \, d t \notag \\
& = \iint_{Q_{\rho}^{\prime}}
 ( u - k )_{+}^{2} \zeta \frac{\partial \zeta}{\partial t} \,
d x \, d t 
+ \sum_{i,j=1}^{n}
\iint_{Q_{\rho}^{\prime}}
 a_{i j} ( u - k )_{+}^{2} \frac{\partial \zeta}{\partial x_{j}}
 \frac{\partial \zeta}{\partial x_{i}} \,
d x \, d t \notag \\
& \hspace*{5ex} \mbox{}
+ \iint_{Q_{\rho}^{\prime}}
 f ( u - k )_{+} \zeta^{2} \,
d x \, d t + \sum_{i=1}^{n}
\iint_{Q_{\rho}^{\prime}}
 f_{i}
 \frac{\partial}{\partial x_{i}} \bigl( ( u - k )_{+} \zeta^{2} \bigr)
 \,
d x \, d t . \label{eq:identity1}
\end{align}
We remark that
\begin{align*}
& \left\lvert
 \iint_{Q_{\rho}^{\prime}}
  f_{i}
  \frac{\partial}{\partial x_{i}} \bigl( ( u - k )_{+} \zeta^{2} \bigr)
  \,
 d x \, d t
\right\rvert \\
& = \Biggl\lvert
 \iint_{Q_{\rho}^{\prime} \cap \{ u(x,t) > k \}}
  f_{i} \zeta
  \frac{\partial}{\partial x_{i}} \bigl( ( u - k )_{+} \zeta \bigr) \,
 d x \, d t \\
 & \hspace*{5ex} \mbox{}
 + \iint_{Q_{\rho}^{\prime} \cap \{ u(x,t) > k \}}
  f_{i} ( u - k )_{+} \zeta
  \frac{\partial \zeta}{\partial x_{i}} \,
 d x \, d t
\Biggr\rvert \displaybreak[1] \\
& \leq \varepsilon_{1}
\iint_{Q_{\rho}^{\prime} \cap \{ u(x,t) > k \}}
 \left\lvert
  \frac{\partial}{\partial x_{i}} \bigl( ( u - k )_{+} \zeta \bigr)
 \right\rvert^{2} \,
d x \, d t \\
& \hspace*{5ex} \mbox{}
+ \frac{1}{\varepsilon_{1}}
\iint_{Q_{\rho}^{\prime} \cap \{ u(x,t) > k \}}
 \lvert f_{i} \zeta \rvert^{2} \, 
d x \, d t \\
& \hspace*{5ex} \mbox{}
+ \iint_{Q_{\rho}^{\prime} \cap \{ u(x,t) > k \}}
 \lvert f_{i} \zeta \rvert^{2} \, 
d x \, d t
+ \iint_{Q_{\rho}^{\prime} \cap \{ u(x,t) > k \}}
 ( u - k )_{+}^{2}
 \left\lvert \frac{\partial \zeta}{\partial x_{i}} \right\rvert^{2} \,
d x \, d t .
\end{align*}
Hence, by (\ref{eq:aijxiixij}) and (\ref{eq:identity1}), we have
\begin{align*}
& \frac{1}{2} \int_{B_{\rho} ( x_{0} )}
 ( u - k )_{+}^{2} \zeta^{2} \,
d x \biggr|_{t = t^{\prime}}
+ \lambda
\iint_{Q_{\rho}^{\prime}}
 \bigl\lvert \nabla \bigl( ( u - k )_{+} \zeta \bigr) \bigr\rvert^{2} \,
d x \, d t \\
& \leq \frac{1}{2} \int_{B_{\rho} ( x_{0} )}
 ( u - k )_{+}^{2} \zeta^{2} \,
d x \biggr|_{t = t^{\prime}} \\
& \hspace*{5ex} \mbox{}
+ \sum_{i,j=1}^{n}
\iint_{Q_{\rho}^{\prime}}
 a_{i j}
 \frac{\partial}{\partial x_{j}} \bigl( ( u - k )_{+} \zeta \bigr)
 \frac{\partial}{\partial x_{i}} \bigl( ( u - k )_{+} \zeta \bigr) \,
d x \, d t \displaybreak[1] \\
& = \iint_{Q_{\rho}^{\prime}}
 ( u - k )_{+}^{2} \zeta \frac{\partial \zeta}{\partial t} \,
d x \, d t 
+ \sum_{i,j=1}^{n}
\iint_{Q_{\rho}^{\prime}}
 a_{i j} ( u - k )_{+}^{2} \frac{\partial \zeta}{\partial x_{j}}
 \frac{\partial \zeta}{\partial x_{i}} \,
d x \, d t \\
& \hspace*{5ex} \mbox{}
+ \iint_{Q_{\rho}^{\prime}}
 f ( u - k )_{+} \zeta^{2} \,
d x \, d t + \sum_{i=1}^{n}
\iint_{Q_{\rho}^{\prime}}
 f_{i}
 \frac{\partial}{\partial x_{i}} \bigl( ( u - k )_{+} \zeta^{2} \bigr)
 \,
d x \, d t \displaybreak[1] \\
& \leq \left\lVert
 \frac{\partial \zeta}{\partial t}
\right\rVert_{L^{\infty} ( Q_{\rho} )}
\iint_{Q_{\rho}^{\prime}}
 ( u - k )_{+}^{2} \,
d x \, d t 
+ \Lambda \lVert \nabla \zeta \rVert_{L^{\infty} ( Q_{\rho} )}^{2}
\iint_{Q_{\rho}^{\prime}}
 ( u - k )_{+}^{2} \,
d x \, d t \\
& \hspace*{5ex} \mbox{}
+ \iint_{Q_{\rho}^{\prime}}
 f ( u - k )_{+} \zeta^{2} \,
d x \, d t 
+ \varepsilon_{1}
\iint_{Q_{\rho}^{\prime}}
 \left\lvert
  \nabla \bigl( ( u - k )_{+} \zeta \bigr)
 \right\rvert^{2} \,
d x \, d t \\
& \hspace*{5ex} \mbox{}
+ \left( \frac{1}{\varepsilon_{1}} + 1 \right)
\iint_{Q_{\rho}^{\prime} \cap \{ u(x,t) > k \}}
 \sum_{i=1}^{n} \lvert f_{i} \rvert^{2} \, 
d x \, d t \\
& \hspace*{5ex} \mbox{}
+ n \lVert \nabla \zeta \rVert_{L^{\infty} ( Q_{\rho} )}^{2}
\iint_{Q_{\rho}^{\prime}}
 ( u - k )_{+}^{2} \,
d x \, d t ,
\end{align*}
that is,
\begin{align*}
& \frac{1}{2} \int_{B_{\rho} ( x_{0} )}
 ( u - k )_{+}^{2} \zeta^{2} \,
d x \biggr|_{t = t^{\prime}}
+ ( \lambda - \varepsilon_{1} )
\iint_{Q_{\rho}^{\prime}}
 \bigl\lvert \nabla \bigl( ( u - k )_{+} \zeta \bigr) \bigr\rvert^{2} \,
d x \, d t \\
& \leq ( \Lambda + n ) \left(
 \left\lVert
  \frac{\partial \zeta}{\partial t}
 \right\rVert_{L^{\infty} ( Q_{\rho} )}
 + \lVert \nabla \zeta \rVert_{L^{\infty} ( Q_{\rho} )}^{2}
\right)
\iint_{Q_{\rho}^{\prime}}
 ( u - k )_{+}^{2} \,
d x \, d t \\
& \hspace*{5ex} \mbox{}
+ \left( \frac{1}{\varepsilon_{1}} + 1 \right)
\iint_{Q_{\rho}^{\prime} \cap \{ u(x,t) > k \}}
 \sum_{i=1}^{n} \lvert f_{i} \rvert^{2} \, 
d x \, d t
+ \iint_{Q_{\rho}^{\prime}}
 f ( u - k )_{+} \zeta^{2} \,
d x \, d t.
\end{align*}
Taking the supremum of the inequality over 
$( t_{0} - \rho^{2} , t_{0} ]$
with respect to $t^{\prime}$, we have
\begin{align}
& \max \left\{
 \frac{1}{2}
 \lVert ( u - k )_{+} \zeta \rVert_{
  L^{\infty} ( t_{0} - \rho^{2} , t_{0} ; L^{2} ( B_{\rho} ( x_{0} ) ) )
 }^{2} , \
 ( \lambda - \varepsilon_{1} )
 \bigl\lVert
  \nabla \bigl( ( u - k )_{+} \zeta \bigr)
 \bigr\rVert_{L^{2} ( Q_{\rho} )}^{2}
\right\} \notag \\
& \leq 
( \Lambda + n ) \left(
 \left\lVert
  \frac{\partial \zeta}{\partial t}
 \right\rVert_{L^{\infty} ( Q_{\rho} )}
 + \lVert \nabla \zeta \rVert_{L^{\infty} ( Q_{\rho} )}^{2}
\right)
\iint_{Q_{\rho}}
 ( u - k )_{+}^{2} \,
d x \, d t \notag \\
& \hspace*{5ex} \mbox{}
+ \left( \frac{1}{\varepsilon_{1}} + 1 \right)
\iint_{Q_{\rho} \cap \{ u(x,t) > k \}}
 \sum_{i=1}^{n} \lvert f_{i} \rvert^{2} \,
d x \, d t
+ \iint_{Q_{\rho}}
 f ( u - k )_{+} \zeta^{2} \,
d x \, d t. \label{eq:proofDG1}
\end{align}
Now we estimate
the last two terms in the right-hand side of (\ref{eq:proofDG1}).
First we obtain
\begin{equation}\label{eq:proofDG2}
\iint_{Q_{\rho} \cap \{ u(x,t) > k \}}
 \lvert f_{i} \rvert^{2} \,
d x \, d t
\leq \bigl\lvert
 Q_{\rho} \cap \{ u(x,t) > k \}
\bigr\rvert^{1 - 2 / p}
\lVert f_{i} \rVert_{L^{p} ( Q_{\rho} )}^{2}
\end{equation}
by H{\"o}lder's inequality.
Now we estimate
$\iint_{Q_{\rho}} f ( u - k )_{+} \zeta^{2} \, d x \, d t$.
We first recall
\begin{align*}
& \lVert ( u - k )_{+} \zeta \rVert_{L^{2(n+2)/n} ( Q_{\rho} )} \\
& \leq C_{1} \left(
 \lVert ( u - k )_{+} \zeta \rVert_{
  L^{\infty} ( t_{0} - \rho^{2} , t_{0} ; L^{2} ( B_{\rho} ( x_{0} ) ) )
 }
 + \bigl\lVert
  \nabla \bigl( ( u - k )_{+} \zeta \bigr)
 \bigr\rVert_{L^{2} ( Q_{\rho} )}
\right)
\end{align*}
by Lemma~\ref{lemma:SobolevV2},
where $C_{1} > 0$ depends only on $n$.
Then, by this inequality, 
H{\"o}lder's inequality and Young's inequality,
we have
\begin{align}
& \iint_{Q_{\rho}} f ( u - k )_{+} \zeta^{2} \, d x \, d t
\notag \\
& \leq
\lVert f \zeta \rVert_{
 L^{2(n+2)/(n+4)} ( Q_{\rho} \cap \{ u(x,t) > k \} )
}
\lVert ( u - k )_{+} \zeta \rVert_{L^{2(n+2)/n} ( Q_{\rho} )} 
\notag \displaybreak[1] \\
& \leq \varepsilon_{2}
\lVert ( u - k )_{+} \zeta \rVert_{L^{2(n+2)/n} ( Q_{\rho} )}^{2}
+ \frac{1}{\varepsilon_{2}}
\lVert f \zeta \rVert_{
 L^{2(n+2)/(n+4)} ( Q_{\rho} \cap \{ u(x,t) > k \} )
}^{2} \notag \displaybreak[1] \\
& \leq 2 \varepsilon_{2} C_{1}^{2}
\max \left\{
 \lVert ( u - k )_{+} \zeta \rVert_{
  L^{\infty} ( t_{0} - \rho^{2} , t_{0} ; L^{2} ( B_{\rho} ( x_{0} ) ) )
 }^{2} , \
 \bigl\lVert
  \nabla \bigl( ( u - k )_{+} \zeta \bigr)
 \bigr\rVert_{L^{2} ( Q_{\rho} )}^{2}
\right\} \notag \\
& \hspace*{5ex} \mbox{}
+ \frac{1}{\varepsilon_{2}}
\lVert f \rVert_{
 L^{\frac{p(n+2)}{n+2+p}} ( Q_{\rho} )
}^{2}
\bigl\lvert
 Q_{\rho} \cap \{ u(x,t) > k \}
\bigr\rvert^{1 - 2 / p} \label{eq:proofDG3}
\end{align}
because 
$2(n+2)/(n+4) < p(n+2)/(n+2+p)$. 
By (\ref{eq:proofDG1}), (\ref{eq:proofDG2}) and (\ref{eq:proofDG3}),
we obtain the estimate (\ref{eq:DG}). 
\end{proof}

By the same argument, we obtain the following lemma for $v_{-} (x) := \max \{ - v(x), 0 \}$.

\renewcommand{\thelemmaprime}%
             {\mathversion{bold}\ref{lemma:DG}$^{\prime}$}
\begin{lemmaprime}
Under the same assumption as in Lemma~\ref{lemma:DG},
a solution $u$ to the parabolic equation
{\rm (\ref{eq:parabolic})} satisfies
\begin{align}
& \lVert ( u - k )_{-} \zeta \rVert_{
 L^{\infty} ( t_{0} - \rho^{2} , t_{0} ; L^{2} ( B_{\rho} ( x_{0} ) ) )
}^{2}
+ \bigl\lVert
 \nabla \bigl( ( u - k )_{-} \zeta \bigr)
\bigr\rVert_{L^{2} ( Q_{\rho} )}^{2} \notag \\
& \leq C_{2} \Biggl[
 \left(
  \left\lVert
   \frac{\partial \zeta}{\partial t}
  \right\rVert_{L^{\infty} ( Q_{\rho} )} 
  + \lVert \nabla \zeta \rVert_{L^{\infty} ( Q_{\rho} )}^{2}
 \right) \lVert ( u - k )_{-} \rVert_{L^{2} ( Q_{\rho} )}^{2} \notag \\
 & \hspace*{10ex} \mbox{}
 + F_{0, \rho}^{2}
 \bigl\lvert
  Q_{\rho}  \cap \{ u(x,t) < k \}
 \bigr\rvert^{1 - 2 / p}
\Biggr] \label{eq:DGprime}
\end{align}
for any $k \in \mathbb{R}$,
where we define $F_{0, \rho}$ as {\rm (\ref{eq:F0rho})},
and $C_{2} > 0$ depends only on $n, \Lambda$ and $\lambda$.
\end{lemmaprime}
The estimate (\ref{eq:LinftyL2}) easily follows from Lemmas~\ref{lemma:DG} and \ref{lemma:DG}$^{\prime}$.
Our next task is to prove the estimate (\ref{eq:Linfty}).
We start by giving a technical lemma which will be used later.
\begin{lemma}\label{lemma:ym}
Let $\widetilde{C} > 0$, $b > 1$ and $\varepsilon > 0$.
If a sequence $\{ y_{m} \}_{m=0}^{\infty}$ satisfies
\begin{equation}\label{eq:ymcond}
y_{0} \leq \theta_{0}
:= \widetilde{C}^{- 1 / \varepsilon} b^{- 1 / \varepsilon^{2}}
\mbox{ and }
0 \leq y_{m+1} \leq \widetilde{C} b^{m} y_{m}^{1 + \varepsilon} ,
\end{equation}
then
\[
\lim_{m \rightarrow \infty} y_{m} = 0
\]
holds.
\end{lemma}
\begin{proof}
We show
\begin{equation}\label{eq:ymr}
y_{m} \leq \frac{\theta_{0}}{r^{m}} , \
m = 0, 1, 2, \ldots
\end{equation}
by inductive method,
where we will determine $r > 1$ later.
By assumption, (\ref{eq:ymr}) with $m=0$ holds.
Hence we now assume (\ref{eq:ymr}) holds,
and show (\ref{eq:ymr}) for $m+1$.
By the assumption (\ref{eq:ymcond})
and the induction hypothesis, we have
\[
y_{m+1} \leq \widetilde{C} b^{m} y_{m}^{1 + \varepsilon}
\leq \widetilde{C} b^{m}
\left( \frac{\theta_{0}}{r^{m}} \right)^{1 + \varepsilon}
= \frac{\theta_{0}}{r^{m+1}} \widetilde{C}
b^{m} \frac{\theta_{0}^{\varepsilon}}{r^{m \varepsilon - 1}} .
\]
Now we take $r = b^{1 / \varepsilon}$. Then we have
\[
y_{m+1} \leq
\frac{\theta_{0}}{r^{m+1}} \widetilde{C}
b^{m} \frac{\theta_{0}^{\varepsilon}}{r^{m \varepsilon - 1}}
= \frac{\theta_{0}}{r^{m+1}} \widetilde{C}
r \theta_{0}^{\varepsilon}
= \frac{\theta_{0}}{r^{m+1}} ,
\]
which is (\ref{eq:ymr}) for $m+1$.
\end{proof}
Now we are now ready to show the estimate (\ref{eq:Linfty}).
The estimate easily follows if we have the following lemma.
\begin{lemma}\label{lemma:Linfty}
Let $p > n + 2$.
Then a solution $u$ to {\rm (\ref{eq:parabolic})}
satisfies the estimate
\[
\lVert u \rVert_{L^{\infty} ( Q_{\rho} )}
\leq C_{\rho}
\left(
 \lVert u \rVert_{L^{2} ( Q_{2 \rho} )}
 + F_{0, 2 \rho}
\right) ,
\]
where we define $F_{0, 2 \rho}$ by {\rm (\ref{eq:F0rho})},
and $C_{\rho} > 0$ depends only on $n, \lambda , \Lambda , p$
and $\rho$.
\end{lemma}
\begin{proof}
First of all a letter $C$ denotes a general constant
depending only on $n, \Lambda , \lambda$ and $p$.
Now, let
$\rho_{m} := ( 1 + 2^{-m} ) \rho$ and
$k_{m} = k ( 2 - 2^{-m} )$
for $m = 0, 1, 2, \ldots$,
where we will determine $k > 0$ later.
For $m = 0, 1, 2, \ldots$,
we take cut-off functions
\begin{math}
\zeta_{m} \in C^{\infty} ( Q_{\rho_{m}} )
\end{math}
which satisfy
\begin{align*}
& 0 \leq \zeta_{m} \leq 1 \mbox{ in } Q_{\rho_{m}} , 
\displaybreak[1] \\
& \zeta_{m} = \left\{
 \begin{aligned}
 & 1 \mbox{ in } Q_{\rho_{m+1}} , \\
 & 0 \mbox{ in }
 Q_{\rho_{m}} \setminus Q_{( \rho_{m} + \rho_{m+1} ) / 2} ,
 \end{aligned}
\right. \displaybreak[1] \\
& \left\lVert
 \frac{\partial \zeta_{m}}{\partial t}
\right\rVert_{L^{\infty} ( Q_{\rho_{m}} )}
+ \lVert \nabla \zeta_{m} \rVert_{L^{\infty} ( Q_{\rho_{m}} )}^{2}
\leq \frac{C}{( \rho_{m} - \rho_{m+1} )^{2}} .
\end{align*}
We remark that 
$\zeta_{m} = 0$
on 
\begin{math}
B_{\rho_{m}} ( x_{0} ) \times \{ t_{0}- \rho^{2} \}
\cup \partial B_{\rho_{m}} ( x_{0} ) 
\times ( t_{0} - \rho^{2} , t_{0} )
\end{math}
in particular. 
By Lemmas~\ref{lemma:SobolevV2} and \ref{lemma:DG}, we have
\begin{align}
& \lVert
 ( u - k_{m+1} )_{+} \zeta_{m}
\rVert_{L^{2(n+2)/n} ( Q_{\rho_{m}} )}^{2} \notag \\
& \leq C \Bigl(
 \lVert ( u - k_{m+1} )_{+} \zeta_{m} \rVert_{
  L^{\infty} (
   t_{0} - \rho_{m}^{2} , t_{0} ; L^{2} ( B_{\rho_{m}} ( x_{0} ) )
  )
 }^{2} \notag \\
 & \hspace*{10ex} \mbox{}
 + \bigl\lVert
  \nabla \bigl( ( u - k_{m+1} )_{+} \zeta_{m} \bigr)
 \bigr\rVert_{L^{2} ( Q_{\rho_{m}} )}^{2}
\Bigr) \notag \displaybreak[1] \\
& \leq C \Biggl[
 \left(
  \left\lVert
   \frac{\partial \zeta_{m}}{\partial t}
  \right\rVert_{L^{\infty} ( Q_{\rho_{m}} )}
  + \lVert \nabla \zeta_{m} \rVert_{L^{\infty} ( Q_{\rho_{m}} )}^{2}
 \right)
 \lVert ( u - k_{m+1} )_{+} \rVert_{L^{2} ( Q_{\rho_{m}} )}^{2} 
 \notag \\
 & \hspace*{10ex} \mbox{}
 + F_{0, \rho_{m}}^{2}
 \bigl\lvert
  Q_{\rho_{m}}  \cap \{ u(x,t) > k_{m+1} \}
 \bigr\rvert^{1 - 2 / p}
\Biggr] \notag \displaybreak[1] \\
& \leq C \Biggl[
 \frac{2^{2 m}}{\rho^{2}}
 \lVert ( u - k_{m+1} )_{+} \rVert_{L^{2} ( Q_{\rho_{m}} )}^{2}
 + F_{0, 2 \rho}^{2}
 \bigl\lvert A_{m} ( k_{m+1} ) \bigr\rvert^{1 - 2 / p}
\Biggr] , \label{eq:a001}
\end{align}
where $A_{m} (l) := Q_{\rho_{m}} \cap \{ u(x,t) > l \}$
for $l \in \mathbb{R}$.
Now we take $k > 0$ as
\begin{equation}\label{eq:k1}
k \geq \rho^{1 - (n+2) / p}  F_{0, 2 \rho} .
\end{equation}
Then we have
\begin{align*}
& \lVert
 ( u - k_{m+1} )_{+} \zeta_{m}
\rVert_{L^{2(n+2)/n} ( Q_{\rho_{m}} )}^{2} \\
& \leq C \Biggl[
 \frac{2^{2 m}}{\rho^{2}}
 \lVert ( u - k_{m+1} )_{+} \rVert_{L^{2} ( Q_{\rho_{m}} )}^{2}
 + \frac{k^{2}}{\rho^{2 ( 1 - (n+2) / p )}}
 \bigl\lvert A_{m} ( k_{m+1} ) \bigr\rvert^{1 - 2 / p}
\Biggr]
\end{align*}
by the estimate (\ref{eq:a001}).
By defining
\begin{math}
\varphi_{m} := \lVert
 ( u - k_{m} )_{+}
\rVert_{L^{2} ( Q_{\rho_{m}} )}^{2},
\end{math}
we have
\begin{align}
\varphi_{m+1}
& = \lVert
 ( u - k_{m+1} )_{+} \zeta_{m}
\rVert_{L^{2} ( Q_{\rho_{m+1}} )}^{2} 
\leq \lVert
 ( u - k_{m+1} )_{+} \zeta_{m}
\rVert_{L^{2} ( Q_{\rho_{m}} )}^{2} \notag \\
& \leq \lvert A_{m} ( k_{m+1} ) \rvert^{2/(n+2)}
\lVert
 ( u - k_{m+1} )_{+} \zeta_{m}
\rVert_{L^{2(n+2)/n} ( Q_{\rho_{m}} )}^{2} \notag 
\displaybreak[1] \\
& \leq C \lvert A_{m} ( k_{m+1} ) \rvert^{2/(n+2)} \notag \\
& \hspace*{3ex} \mbox{} \times
\Biggl[
 \frac{2^{2 m}}{\rho^{2}}
 \lVert ( u - k_{m+1} )_{+} \rVert_{L^{2} ( Q_{\rho_{m}} )}^{2}
 + \frac{k^{2}}{\rho^{2 ( 1 - (n+2) / p )}}
 \bigl\lvert A_{m} ( k_{m+1} ) \bigr\rvert^{1 - 2 / p}
\Biggr] \notag \displaybreak[1] \\
& \leq C \lvert A_{m} ( k_{m+1} ) \rvert^{2/(n+2)}
\Biggl[
 \frac{2^{2 m}}{\rho^{2}} \varphi_{m}
 + \frac{k^{2}}{\rho^{2 ( 1 - (n+2) / p )}}
 \bigl\lvert A_{m} ( k_{m+1} ) \bigr\rvert^{1 - 2 / p}
\Biggr] , \label{eq:a002}
\end{align}
where we used H{\"o}lder's inequality
and the estimate
\[
\lVert ( u - k_{m+1} )_{+} \rVert_{L^{2} ( Q_{\rho_{m}} )}^{2}
\leq \lVert ( u - k_{m} )_{+} \rVert_{L^{2} ( Q_{\rho_{m}} )}^{2}
= \varphi_{m} .
\]
On the other hand, we have
\begin{align*}
\varphi_{m}
& = \lVert ( u - k_{m} )_{+} \rVert_{L^{2} ( Q_{\rho_{m}} )}^{2}
\geq \iint_{A_{m} ( k_{m+1} )}
 ( u - k_{m} )_{+}^{2} \,
d x \, d t \\
& \geq \iint_{A_{m} ( k_{m+1} )}
 ( k_{m+1} - k_{m} )_{+}^{2} \,
d x \, d t
= \frac{k^{2}}{2^{2m+2}}
\lvert A_{m} ( k_{m+1} ) \rvert ,
\end{align*}
that is,
\begin{equation}\label{eq:a003}
\lvert A_{m} ( k_{m+1} ) \rvert
\leq \frac{2^{2m+2}}{k^{2}} \varphi_{m} .
\end{equation}
By (\ref{eq:a002}) and (\ref{eq:a003}), we have
\begin{align}
\varphi_{m+1}
& \leq C 2^{2 m \left( 1 + \frac{2}{n+2} \right)} \notag \\
& \hspace*{3ex} \mbox{} \times \left[
 \rho^{-2} k^{- \frac{4}{n+2}} \varphi_{m}^{1 + \frac{2}{n+2}}
 + \rho^{- 2 \left( 1 - \frac{n+2}{p} \right)}
 k^{- \left( \frac{4}{n+2} - \frac{4}{p} \right)}
 \varphi_{m}^{1 + \frac{2}{n+2} - \frac{2}{p}}
\right] .\label{eq:a004}
\end{align}
We now take $k$ as
\begin{equation}\label{eq:k2}
k \geq \left(
 \frac{1}{\lvert Q_{2 \rho} \rvert}
 \iint_{Q_{2 \rho}} u^{2} \, d x \, d t
\right)^{1/2} .
\end{equation}
Then we have
\[
\varphi_{m}
\leq \iint_{Q_{\rho_{m}}} u^{2} \, d x \, d t
\leq \iint_{Q_{2 \rho}} u^{2} \, d x \, d t
\leq \lvert Q_{2 \rho} \rvert k^{2} ,
\]
that is,
\[
\varphi_{m}^{2/p} \leq 
\lvert Q_{2 \rho} \rvert^{2/p} k^{4/p} .
\]
By this inequality and (\ref{eq:a004}), we have
\begin{align}
\varphi_{m+1}
& \leq C 2^{2 m \left( 1 + \frac{2}{n+2} \right)}
\varphi_{m}^{1 + \frac{2}{n+2} - \frac{2}{p}}
\left[
 \rho^{-2} k^{- \frac{4}{n+2}}
 \varphi_{m}^{2/p}
 + \rho^{- 2 \left( 1 - \frac{n+2}{p} \right)}
 k^{- \left( \frac{4}{n+2} - \frac{4}{p} \right)}
\right] \notag \\
& \leq C 2^{2 m \left( 1 + \frac{2}{n+2} \right)}
\varphi_{m}^{1 + \frac{2}{n+2} - \frac{2}{p}} \notag \\
& \hspace*{3ex} \mbox{} \times 
\left[
 \rho^{-2} k^{- \frac{4}{n+2}}
 \lvert Q_{2 \rho} \rvert^{2/p} k^{4/p}
 + \rho^{- 2 \left( 1 - \frac{n+2}{p} \right)}
 k^{- \left( \frac{4}{n+2} - \frac{4}{p} \right)}
\right] \notag \displaybreak[1] \\
& = C 2^{2 m \left( 1 + \frac{2}{n+2} \right)}
\rho^{- 2 \left( 1 - \frac{n+2}{p} \right)}
k^{- \frac{4}{n+2} \left( 1 - \frac{n+2}{p} \right)}
\varphi_{m}^{1 + \frac{2}{n+2} - \frac{2}{p}} .
\label{eq:a005}
\end{align}
Now we denote
$y_{m} := k^{-2} \lvert Q_{2 \rho} \rvert^{-1} \varphi_{m}$.
Then by (\ref{eq:a005}), we have
\begin{equation}\label{eq:a006}
y_{m+1}
\leq C 2^{2 m \left( 1 + \frac{2}{n+2} \right)}
y_{m}^{1 + \left( \frac{2}{n+2} - \frac{2}{p} \right)} ,
\end{equation}
which is the second condition of (\ref{eq:ymcond})
with
\begin{equation}\label{eq:a008}
\widetilde{C} = C, \
b = 2^{2 \left( 1 + \frac{2}{n+2} \right)}
\mbox{ and } \varepsilon = \frac{2}{n+2} - \frac{2}{p} .
\end{equation}
Then
$\lim_{m \rightarrow \infty} y_{m} = 0$
if
\begin{equation}\label{eq:a007}
y_{0} \leq C^{- 1 / \varepsilon} b^{- 1 / \varepsilon^{2}}
=: \theta_{0}
\end{equation}
by Lemma~\ref{lemma:ym},
where $b$ and $\varepsilon$ are defined by (\ref{eq:a008})
and $C$ is the constant $C$ in (\ref{eq:a006}).
We remark that the condition (\ref{eq:a007}) is equivalent to
\begin{equation}\label{eq:a009}
\lVert ( u - k )_{+} \rVert_{L^{2} ( Q_{2 \rho} )}^{2}
\leq \theta_{0} k^{2} \lvert Q_{2\rho} \rvert .
\end{equation}
Now we take $k$ as
\begin{equation}\label{eq:k3}
k^{2} \geq \frac{1}{\theta_{0} \lvert Q_{2\rho} \rvert}
\lVert u \rVert_{L^{2} ( Q_{2 \rho} )}^{2} .
\end{equation}
Then the condition (\ref{eq:a009}), i.e.
the condition (\ref{eq:a007}) is satisfied.

Summing up, if we take $k$ such that
the conditions (\ref{eq:k1}), (\ref{eq:k2}) and (\ref{eq:k3}) are satisfied,
then we have
\begin{math}
\lim_{m \rightarrow \infty} y_{m} = 0
\end{math}.
On the other hand, since
\begin{align*}
y_{m} 
& = \frac{1}{k^{2} \lvert Q_{2 \rho} \rvert} \varphi_{m}
= \frac{1}{k^{2} \lvert Q_{2 \rho} \rvert}
\lVert
 ( u - k_{m} )_{+}
\rVert_{L^{2} ( Q_{\rho_{m}} )}^{2} \\
& \rightarrow \frac{1}{k^{2} \lvert Q_{2 \rho} \rvert}
\lVert
 ( u - 2 k )_{+}
\rVert_{L^{2} ( Q_{\rho} )}^{2}
\mbox{ as } m \rightarrow \infty .
\end{align*}
Then we have
$\lVert ( u - 2 k )_{+} \rVert_{L^{2} ( Q_{\rho} )}^{2} = 0$,
that is,
\begin{equation}\label{eq:Linftyresult}
u \leq 2 k \mbox{ a.e.\ in } Q_{\rho} .
\end{equation}
Now we take $k$ as
\[
k = \frac{1}{\sqrt{\theta_{0} \lvert Q_{2 \rho} \rvert}}
\lVert u \rVert_{L^{2} ( Q_{2 \rho})}
+ \rho^{1 - (n+2) / p} F_{0, 2 \rho} ,
\]
which satisfies
the conditions (\ref{eq:k1}), (\ref{eq:k2}) and (\ref{eq:k3}).
Hence we have (\ref{eq:Linftyresult}), which is
\[
\sup_{Q_{\rho}} u \leq C_{\rho}
\left(
 \lVert u \rVert_{L^{2} ( Q_{2 \rho} )} + F_{0, 2 \rho}
\right) .
\]
Replacing Lemma~\ref{lemma:DG}
by Lemma~\ref{lemma:DG}$^{\prime}$
and doing the same argument, we can obtain
\[
- u \leq C_{\rho}
\left(
 \lVert u \rVert_{L^{2} ( Q_{2 \rho} )} + F_{0, 2 \rho}
\right) \mbox{ in } Q_{\rho}
\]
and thus the proof has been completed.
\end{proof}
\section*{Acknowledgments}
The first author is supported
by Postdoctoral Fellowship for Foreign Researchers
from the Japan Society for the Promotion of Science.
The third and fourth authors are supported
by Research Fellowships of the Japan Society
for the Promotion of Science for Young Scientists and Grant-in-Aid for Scientific Research (B)
(No.\ 22340023) of Japan Society for Promotion of Science, respectively.

\end{document}